\documentclass[11pt, a4paper]{article}  
\usepackage[hidelinks, pdftex]{hyperref}
\usepackage[numbers]{natbib}
\usepackage{algorithm}
\usepackage{algpseudocode}
\usepackage[utf8]{inputenc}
\usepackage{amsmath}
\usepackage{amssymb}
\usepackage{amsthm}
\usepackage{color}
\usepackage{cleveref}
\usepackage{bbm}
\usepackage{graphicx}
\usepackage{enumitem}
\usepackage{subcaption}

\def \C {\mathbb{C}}
\def \E {\mathbb{E}}

\def \P {\mathbb{P}}
\def \R {\mathbb{R}}
\def \S {\mathbb{S}}

\def \tot {\mathrm{tot}}

\def \F {\mathrm{F}}
\def \KL {\mathrm{KL}}

\def \calC {\mathcal{C}}

\def \calP {\mathcal{P}}

\def \calS {\mathcal{S}}

\def \rank {\mathrm{rank}}
\def \range {\mathrm{range}}
\def \tr {\mathrm{tr}}

\def \d {\mathrm{d}}

\def \rd {\mathrm{d}}

\makeatletter
\renewcommand*\env@matrix[1][*\c@MaxMatrixCols c]{%
  \hskip -\arraycolsep
  \let\@ifnextchar\new@ifnextchar
  \array{#1}}
\makeatother

\newtheorem{theorem}{Theorem}

\newtheorem{lemma}{Lemma}

\title{A Langevin sampler for quantum tomography}
\author{Tameem Adel\thanks{National Physical Laboratory, Teddington, TW11 0LW, UK.}\ , Abhishek Agarwal\footnotemark[1]\ , Stéphane Chrétien\thanks{Laboratoire ERIC, Université Lyon 2, 5 Av. Pierre Mendès
France, Bron, 69500, France.}\ , Estelle Massart\thanks{ICTEAM, UCLouvain, Avenue Georges Lemaître 4, 1348 Louvain-la-Neuve,  Belgium.}\ \thanks{This work was initiated and partly undertaken while this author was a postdoctoral researcher at the Mathematical Institute, University of Oxford, UK.}\ \thanks{Corresponding author: \url{estelle.massart@uclouvain.be}.} , \\ Danila Mokeev\footnotemark[3]\ , Ivan Rungger\footnotemark[1]\ \thanks{Department of Computer Science, Royal Holloway, University of London, Egham, TW20 0EX, UK.}\ , Andrew Thompson\footnotemark[1] }
\date{\today}

\begin{document}
\maketitle

\paragraph{Abstract:} Quantum tomography involves obtaining a full classical description of a prepared quantum state from experimental results. We propose a Langevin sampler for quantum tomography, that relies on a new formulation of Bayesian quantum tomography exploiting the Burer-Monteiro factorization of Hermitian positive-semidefinite matrices. If the rank of the target density matrix is known, this formulation allows us to define a posterior distribution that is only supported on matrices whose rank is upper-bounded by the rank of the target density matrix. Conversely, if the target rank is unknown, any upper bound on the rank can be used by our algorithm, and the rank of the resulting posterior mean estimator is further reduced by the use of a low-rank promoting prior density. This prior density is a complex extension of the one proposed in [Annales de l'Institut Henri Poincar\'e Probability and Statistics, 56(2):1465–1483, 2020]. We derive a PAC-Bayesian bound on our proposed estimator that matches the best bounds available in the literature, and we show numerically that it leads to strong scalability improvements compared to existing techniques when the rank of the density matrix is known to be small.

\section{Introduction}

The ability to prepare, control, and maintain quantum states has seen tremendous progress since the mid-90s, and today devices with hundreds of qubits have been developed~\cite{bluvstein2024logical,acharya2024quantum}. With these advancements arises the task of verifying the correctness of quantum states that have been prepared, which can then also enable the characterization of different components of the device~\cite{lall2025reviewcollectionmetricsbenchmarks}. One widely-used approach to performing this characterization involves obtaining a full classical description of a prepared quantum state from experimental results, known as quantum state tomography or simply quantum tomography~\cite{vogel1989}. Methods for performing quantum tomography are widely used for the characterization of quantum computers~\cite{proctor2025benchmarking}, quantum sensors~\cite{farooq2022robust}, and quantum communication networks~\cite{Bouchard2019quantumprocess}. 

Quantum tomography, in general, requires an exponentially-increasing number of experimental measurements with the size of the quantum system due to the exponentially-growing dimension of the Hilbert space. Thus, it is typically only used to characterize quantum states for a small number of qubits. A single qubit is a two-level quantum system, and a full description of an $n$-qubit quantum system is provided by its density matrix $\rho^0$, where  $\rho^0 \in \C^{d \times d}$, with $d = 2^n$. The density matrix can be estimated by performing measurements on an ensemble of copies of $\rho^0$, which are typically obtained by repeatedly re-initializing the system and preparing the quantum state after each measurement. Quantum tomography, then, amounts to finding a density matrix $\rho^0$ that characterizes the unknown quantum state of a 2-level $n$-qubit system, given a set of measurements of repeated state preparation. 

There are a number of different methods that have been developed to reconstruct the density matrix from experimental results; see \cite{gebhart2023learning} for a recent review. While maximum likelihood estimators have attracted much attention \cite{Guta2020}, several works explore a Bayesian formulation for quantum tomography, allowing for a quantification of the state's uncertainty \cite{Blume2010,Huszar2012,Granade2016,Mai2017, Lukens2020,Mai2021}. Quantification of this uncertainty is important since it allows accounting for limited statistical precision in the data and enables distinguishing true features of the quantum state from artificial ones introduced by finite statistics \cite{Suess_2017,agarwal2024modelling,agarwal2025fasttracking}. Given a suitably selected prior density and likelihood function, Bayesian quantum tomography defines a posterior distribution on the set of density matrices, whose expectation provides a natural estimate for the density matrix of the system. However, in practice, computing this expectation requires evaluating a high-dimensional integral, which is computationally challenging. Existing approaches for quantum tomography rely on Markov Chain Monte-Carlo (MCMC) sampling, typically combined with Metropolis-Hastings and possibly enhanced by preconditioned Crank-Nicolson, see \cite{Mai2021}. Here we propose an alternative that relies on Langevin sampling. 

Langevin dynamics are considered a class of MCMC techniques \cite{Neal2011a}. In cases where direct sampling is not straightforward, Langevin dynamics are typically used to obtain sequences of random samples from probability distributions. A general Langevin dynamics-based algorithm proposes new states using evaluations of the gradient of the log-posterior of the target probability density function. 

\subsection{Related works}
The first and most straightforward method for quantum tomography simply inverts the equations that result from Born's rule, which relates the probabilities of measurement outcomes to the classical description of the quantum state given by the density matrix $\rho^0$. Arguably the main drawback of this method, called linear or direct inversion in the literature, is its possible failure in returning an estimator that has the structural properties of a density matrix, calling for an additional projection step onto the set of feasible density matrices \cite{Guta2020}. For this reason, maximum likelihood estimation, which amounts to finding the most likely density matrix for a given set of measurements \cite{Hradil2004}, became the method of choice for quantum tomography.

 Motivated by the possibility of capturing the uncertainty of the estimate, more and more works address a Bayesian formalism for quantum tomography \cite{Jones1991, Buzek1998, Blume2010, Kravtsov2013, Kueng2015}.  For example, the author of \cite{Blume2010} proposed a Bayesian framework for quantum tomography in which the estimated density matrix is then obtained as the mean of the posterior distribution, approximated using Metropolis-Hastings. This method was extended by \cite{Huszar2012}, who proposed an adaptive Bayesian quantum tomography method, endowing Bayesian quantum tomography with a sequential importance sampling strategy which determines adaptively new measurements to use in the experiment. Furthermore, in \cite{Granade2016} the authors handle time-dependent states and propose insightful priors, namely, families of priors built from default (uninformative) priors and that have prescribed mean, chosen, e.g., based on experimental design or previous experimental estimates.

 In the context of quantum computing, one typically prepares pure states which are characterized by rank-1 density matrices \cite{Nielsen_Chuang_2010}. However, the presence of incoherent noise in quantum computers leads to the system being described by mixed states \cite{preskill1998lecture}. In certain cases, such as when the noise affects qubits locally, the resulting quantum state is of low rank \cite{Gross2010} and this low-rank property can be utilized to simplify quantum tomography. For example, \cite{Alquier2013} further project the linear projection estimator onto the set of fixed-rank density matrices, and prove the consistency of their proposed estimator. Alternatively, \cite{Guta2012} construct a sequence of maximum likelihood models with restricted rank, and select among those models the one that achieves the best trade-off between the risk and model complexity measured using the well-known Akaike Information Criterion (AIC) or Bayesian Information Criterion (BIC). Other approaches such as compressed sensing have also been developed for the tomography of low-rank states \cite{Gross2010}. 

For Bayesian quantum tomography, rank information can be accounted for by resorting to a suitable prior density; see for example the  prior used in \cite{Granade2016} that relies on the Ginibre ensemble. Alternatively,  the authors of \cite{Mai2017} rewrite density matrices as a sum of rank 1 terms scaled by factors sampled according to a Dirichlet distribution, which promotes their sparsity. They propose two estimators that are inherited from the Bayesian quantum tomography formalism: the \emph{dens-estimator} (where the likelihood captures the distance to the least-squares estimates of the density matrix), and the \emph{prob-estimator} (where the likelihood captures the distance between the theoretical and empirical state probability vectors, which coincides with the approach used in this paper). The authors compute these estimators using the Metropolis-Hastings algorithm, and derive corresponding PAC-Bayesian bounds. As Metropolis-Hastings is slow for high-dimensional systems, another efficient adaptive Metropolis-Hastings implementation for the prob-estimator is proposed in \cite{Mai2021}, relying on Crank-Nicholson preconditioning. A similar approach was used for the dens-estimator in  \cite{Lukens2020}. 
We finally mention the more recent work \cite{Quek2021} which proposed an adaptive quantum state tomography algorithm, in line with \cite{Huszar2012}, where the standard Bayes' update is performed using a recurrent neural network.

\subsection{Contributions} Exploiting the celebrated Burer-Monteiro factorization of Hermitian positive-semidefinite matrices, we formulate the sampling problem in terms of low-rank factors of the density matrix, and derive a Langevin sampler for the resulting posterior distribution. For this, we propose a new low-rank promoting prior density, that is a complex generalization of the prior proposed in \cite{Dalalyan2020}. We derive PAC-Bayesian bounds in the complete measurement setting for our proposed model, obtaining an identical leading rate (up to constant/logarithmic factor) to the one obtained in \cite{Mai2017,Mai2021}, and we compare our method numerically to the MCMC sampler proposed in \cite{Mai2017}. We show that, while both estimators lead to comparable accuracies, our proposed Langevin sampler converges in fewer iterations than the MCMC sampler from \cite{Mai2017}. Moreover, our numerical results also indicate that, when the rank of the density matrix is known to be small (allowing smaller dimensions for the Burer-Monteiro factors), each iteration of the Langevin sampling algorithm can be computed at a much lower cost than iterations of the sampler proposed in \cite{Mai2017}.

\section{Bayesian quantum tomography} \label{sec:gen_framework}
Here we address the estimation of the density matrix $\rho^0 \in \C^{d \times d}$ of a 2-level $n$-qubit system, implying that $d = 2^n$. The density matrix $\rho^0$ can be decomposed as a convex combination of pure eigenstates,
$$\rho^0=\sum_{i=1}^{d}\lambda_i v_i v_i^{\ast},$$
where the $\{v_i v_i^{\ast}\}$ are the pure eigenstates and where the $\{\lambda_i\}$ are the associated eigenvalues. It follows that $\rho^0$ is a complex Hermitian positive-semidefinite matrix whose trace is equal to one. The set of density matrices of an $n$-qubit system therefore corresponds to the complex spectrahedron:
\begin{equation}  \label{eq: spectra}
    \calC \calS = \{ \rho \in \C^{d \times d} : \rho = \rho^*, \rho \succeq 0, \tr(\rho) = 1 \}, 
\end{equation}
on which we assume a prior density $\mu$ to be available.

Analogously to \cite{Mai2017}, we assume that, for each qubit, one can measure one of the three Pauli observables $\sigma_x$, $\sigma_y$, $\sigma_z$. There are therefore $3^n$ possible experiments, with outcome vector in $\{-1,1\}^n$. Let $R^a$ be the random vector defined as the result of  experiment $a \in \{1, \dots, 3^n\}$,  let $\calP_s^a$ be the measurement operator associated with experiment $a \in \{1, \dots, 3^n\}$ and outcome vector $o(s) \in \{-1,1\}^n$, for $s \in \{1, \dots, 2^n\}$. We write
\[ p_{a,s}^0 := \mathrm{prob}[R^a = o(s)],  \]
the probability of experiment $a$ to yield outcome $s$ which, according to Born's rule, is given by 
 \[ p_{a,s}^0  =\mathrm{trace}(\calP_s^a \cdot \rho^0). \]

While the true probability $p_{a,s}^0$ is typically unavailable, it can be approximated by using measurement outcomes associated to several replications of the system. Assuming that an experiment $a$ is replicated $m_a$ times, the true probability $p_{a,s}^0$ can thus be approximated by the empirical probability 
\begin{equation} \label{eq:empirical_freq}
    \hat p_{a,s} := \frac{1}{m_a} \sum_{i = 1}^{m_a} 1_{R_i^a = s},
\end{equation}
where $R_i^a$ is the corresponding $i$th replication of $R^a$, for $i = 1, \dots, m_a$. These empirical frequencies can be used to derive a likelihood function\footnote{We acknowledge a small abuse of terminology here, as this \emph{likelihood} does not come from a noise model on the measurements, and is therefore referred to as a pseudo-likelihood (and similarly, the resulting posterior a pseudo-posterior) in \cite{Mai2017}.}, that we choose here to be equal to 
\begin{align} \label{eq:likelihood_rho} 
     L: \C^{d \times d} \to \R : \rho  \mapsto  \sum_{s = 1}^{2^n} 
 \sum_{a = 1}^{3^n} ( \hat p_{a,s} - \tr(\calP_s^a \cdot \rho))^2.
\end{align}
Given a suitably selected prior density $\mu$, Bayesian quantum tomography defines a posterior distribution\footnote{As the likelihood does not necessarily come from a noise model on the measurements, the posterior is referred to as a pseudo-posterior in \cite{Mai2017}.} 
\begin{equation*}
    \hat \mu_\lambda(\rho) \propto \exp\left(- \lambda L(\rho)\right) \mu(\rho)
\end{equation*}
over the set of density matrices, for $\lambda >0$ a suitably chosen parameter that balances the likelihood and prior information. The density matrix $\rho^0$ can then be estimated by, e.g., the empirical average of the posterior density\footnote{This choice is aligned with the choice made in \cite{Mai2017, Mai2021,Lukens2020}. Note that the choice of a multimodal prior may motivate the use of the maximum \emph{a posteriori} estimator instead.}: 
\begin{equation} \label{eq:intro_estim}
    \hat \rho_\lambda = \int_{\rho \in \calC \calS} \rho \hat \mu_\lambda(\rho) \d \rho. 
\end{equation}
A possible strategy to compute this estimator is to rely on Langevin sampling, which samples the posterior distribution $\hat \mu_\lambda$ by generating the Markov chain
\begin{equation} \label{eq:langevin_rho}
     \rho_{k+1} = \rho_k - \eta_k \nabla f(\rho_k) + \sqrt{\frac{2\eta_k}{\beta_k}} w_k,
\end{equation} 
starting from some iterate $\rho_0$, where $f(\rho) := - \log(\hat \mu_\lambda(\rho))$ is the (opposite) log-posterior, $\eta_k$ the stepsize, and $\beta_k$ a temperature hyperparameter\footnote{While usually $\beta_k$ is set to one in the Langevin sampling literature (see, e.g., \cite{Dalalyan2020}), this temperature parameter is here introduced for flexibility. In this work, we will consider that $\beta_k$ is fixed for all iterations and learned using cross-validation. This amounts to merely rescaling the noise variance, or, equivalently, rescaling the log-posterior density $f$.}, and $w_k$ is a complex random matrix whose real and complex parts each follow a standard Gaussian distribution.

\section{A Burer-Monteiro formulation for Bayesian quantum tomography}

Note that \eqref{eq:langevin_rho} generates a Markov chain in the space of $d \times d$ density matrices, omitting the fact that the sought density is typically low-rank, leading to a possible reduction of the sampling space dimension. We propose instead to rely on the well-known Burer-Monteiro approach \cite{Burer2003}, and  represent any density matrix $\rho \in \C^{d \times d}$, assumed to be of rank $r$, by a factor $Y \in \C^{d \times r}$ such that $\rho = YY^*$. The advantage of this reformulation is twofold. First, this formulation results in a decrease of the dimension of the parameter space from $d^2$ to $dr$ complex numbers, hence, decreases the dimension of the posterior density from which to sample. Secondly, it naturally accounts for the Hermitian positive-semidefiniteness of the density matrix, as for any $Y \in \C^{d \times r}$, the matrix $YY^*$ is Hermitian positive semidefinite of rank at most $r$. Since, by definition of the trace and Frobenius norm, for all $\rho = YY^*$ there holds
\[  \tr(\rho) = \tr(YY^*) = \|Y\|_\F^2,\]
the unit trace constraint on density matrices is equivalent to requiring that the factors belong to the complex hypersphere 
\[\calC \S^{d \times r} =  \{ Y \in \C^{d \times r} : \| Y \|_\F = 1 \},\]
and the likelihood \eqref{eq:likelihood_rho} becomes  
\begin{align} \label{eq:likelihood} 
      L^{\mathrm{BM}}: \C^{d \times r} \to \R : Y  \mapsto  \sum_{s = 1}^{2^n} 
 \sum_{a = 1}^{3^n} ( \hat p_{a,s} - \tr(\calP_s^a \cdot YY^*))^2,
\end{align}
where the BM subscript stands for ``Burer Monteiro''.  Any prior distribution on $ \C^{d \times r}$ (or, more precisely, on the complex hypersphere $\calC \S^{d \times r}$) will thus naturally lead to a posterior distribution on $ \C^{d \times r}$ that can, in turn,  be used to derive an estimator for the target density matrix $\rho^0 = Y^0 Y^{0*}$. Note however that the use of Burer-Monteiro factorization relies on the assumption that the rank of the target density matrix is known. When the latter is unknown, any upper bound on the rank of the target density matrix can be used (including, when information is not available, the choice $r = d$). In this case, we argue that low-rankness can be further improved by selecting a suitable prior density.

A prior inducing low-rankness in the density matrices was proposed in \cite{Mai2017}, where the density matrices are represented as a product 
\[ \rho = V \Gamma V^*,\]
with $V$ a matrix with unit-norm (not necessarily orthogonal) columns, and $\Gamma$  a diagonal matrix with non-negative entries\footnote{While the non-orthogonality of the columns of $V$ precludes the interpretation of the diagonal entries of $\Gamma$ as eigenvalues of $\rho$, this approach is motivated by the fact that sampling unit-norm vectors is cheaper than sampling matrices with orthonormal columns for the uniform density that the authors are using.} $\gamma_1,\ldots,\gamma_d$. The prior is then chosen as $(\gamma_1, \dots, \gamma_d) \sim \mathcal{D}ir(\alpha_1, \dots, \alpha_d)$, with $\mathcal{D}ir(\alpha_1, \dots, \alpha_d)$ the Dirichlet distribution with parameters $\alpha_1, \dots, \alpha_d >0$, and $V_{:,i} \sim Unif(\mathcal{S}^{d-1})$, with $V_{:,i}$ the $i$th column of $V$ and $Unif(\mathcal{S}^{d-1})$ a uniform prior on the unit complex hypersphere. To promote sparsity, the authors of \cite{Mai2017} choose $\alpha_1 = \dots = \alpha_d = \alpha$ for some $\alpha$ close to zero (e.g., $\alpha_1 = \dots = \alpha_d = 1/d).$

We resort here to an alternative prior that is a complex extension of the spectral scaled Student's $t$-distribution described in \cite{Dalalyan2020}, and whose expressions makes it a natural choice for our Burer-Monteiro formulation:
\begin{equation} \label{eq:prior}
\nu_{\theta}(Y) = C_\theta \det(\theta^2 I_d + YY^*)^{-(2d+r+2)/2},
\end{equation}
for $Y \in \C^{d \times r}$, with $\theta > 0$ a distribution parameter, and $C_\theta = (\int_{\C^{d \times r}} \det(\theta^2 I_d + YY^*)^{-(2d+r+2)/2}\mathrm{d}Y)^{-1}$ a normalizing constant. Similarly as in \cite{Dalalyan2020}, note that 
\[  \nu_{\theta}(Y) \propto \Pi_{i = 1}^r (\theta^2 + \sigma_i^2)^{-(2d+r+2)/2}, \]
with $\sigma_1, \dots, \sigma_r$ the singular values of $Y$. As a result, our proposed prior amounts to assuming that the singular values of $Y$ follow a scaled Student's $t$-distribution, which promotes their sparsity; see, e.g.,  \cite{Dalalyan2012}. It follows that this prior promotes low-rankness of the matrix $YY^*$.

As stated by the next result, which is an immediate extension of \cite[Lemma 1]{Dalalyan2020}, the columns of $Y$ are marginally distributed according to a $d$-variate complex scaled Student's $t$-distribution. 

\begin{lemma} \label{lem:multivariatestudent}
    If $Y$ is a random $d \times r$ complex matrix having as density the function $\nu_{\theta}$, then the column vectors $y_i \in \C^d$ of $Y$, for $i = 1, \dots, r$, follow the $d$-variate complex scaled Student's $t$-distribution $(\sqrt{2/3} \theta)t_{3,d}$. As a consequence, it holds that $\int_{\C^{d\times r}} \|y_i\|^2 \nu_\theta(Y)\mathrm{d}Y = 2\theta^2 d$ for all $i$. 
\end{lemma}

In this work, we rely on this spectral scaled Student's $t$ prior distribution $\nu_{\theta}(Y)$ which, combined with the likelihood \eqref{eq:likelihood}, gives a posterior density 
\begin{equation} \label{eq:BM_posterior}
\hat \nu_{\lambda,\theta}^{\mathrm{BM}}(Y)  \propto \exp(- \lambda L^{\mathrm{BM}}(Y)) \nu_{\theta}(Y),  
\end{equation}
and our density matrix estimator is given by 
\begin{equation} \label{eq:estimator}
    \hat \rho_{\lambda,\theta}^{\mathrm{BM}} := \int_{\C^{d \times r}} Y Y^* \hat \nu_{\lambda,\theta}^{\mathrm{BM}}(Y) \mathrm{d}Y.
\end{equation}

The authors of \cite{Mai2017} derived a PAC-Bayesian bound for their estimator. We extend here their analysis to our alternative prior, showing that this new prior preserves the rate achieved in \cite{Mai2017}. Similarly to \cite{Mai2017}, we assume the complete measurement setting, meaning that each experiment $a \in \{1, \dots, 3^n\}$ is performed $m_a = m$ times with $m$ a given number, leading to $N_{\tot} = m 3^n$ measurements in total. The following result provides a PAC-Bayesian bound for our proposed estimator $\hat \rho_{\lambda,\theta}^{\mathrm{BM}}$ defined in \eqref{eq:estimator}.

\begin{theorem}[PAC-Bayesian bound] \label{thm:pac}
 Let $\lambda = 3m/8$, and let $\bar Y \in \C^{d \times r}$ satisfy $ \|Y^0-\bar Y \|_\F \|Y^0+\bar Y \|_\F \leq 3^{-3n/2} 2^{-n/2}/m$. Let $p$ be the rank of $\bar Y$. For any $\epsilon \in ]0,1[$, there holds with probability $1-\epsilon$
\begin{align*}
    \|\hat \rho_{\lambda, \theta}^{\mathrm{BM}}  - \rho^0 \|^2_\F  &\leq \frac{3}{N_{\tot}} \left(3^{3n/4} 2^{(n+6)/4} (r + \sqrt{r} \|\bar Y\|_\F) + \frac{2r}{m}+1\right) \\
    &+\frac{3^n 8}{ 2^n  N_{\tot}} \left( \log\left(\frac{2}{\epsilon}\right) + 2p (2^{n+1}+r+2) \log \left( 1+ \frac{\| \bar Y \|_2}{\theta} \right) \right).
\end{align*} 
\end{theorem}

While the upper bound in \Cref{thm:pac} depends on an arbitrary matrix $\bar Y$, note that choosing $\bar Y = Y^0$ leads to $p = \rank(\rho^0)$, in which case we recover for large values of $n$ the rate $3^n \rank(\rho_0)/N_{\tot}$ obtained in \cite{Mai2017}, up to constant or logarithmic factors.

\section{Implementation and numerical results}

\subsection{Algorithmic implementation}

To avoid the need to manipulate complex numbers in our implementations, we introduce a change of variables, relying on the well-known vector space isomorphism between $\C^{d \times r}$ and a subset of $\R^{2d \times 2r}$ (see, e.g., \cite{Ollila2012}):
\begin{equation} \label{eq:psi}
    \psi : M^R+i M^I \mapsto \frac{1}{\sqrt{2}} \begin{pmatrix} M^R & -M^I \\ M^I & M^R \end{pmatrix}.
\end{equation}
It can then be readily checked that, for any $Y \in \C^{d \times r}$, 
\begin{equation} \label{eq:_in_pro}
    \psi(YY^*) = \sqrt{2} \psi(Y) \psi(Y)^\top. 
\end{equation}
The next lemma, whose proof is in the Appendix, allows us to fully rewrite the posterior density in terms of real-valued matrices.

\begin{lemma} \label{lem:ComplexToReal}
For all $M,N \in \C^{d \times d}$ Hermitian positive semidefinite, there holds $\tr(MN) =  \tr(\psi(M) \psi(N))$ and $\det(M) = \sqrt{2^d \det(\psi(M))}$.
\end{lemma}

Instead of sampling the posterior density $\hat \nu_{\lambda,\theta}^{\mathrm{BM}}$ over the set of complex $d \times r$ matrices, we sample the posterior $\tilde \nu_{\lambda,\theta}^{\mathrm{BM}}(\tilde Y) := \hat \nu_{\lambda,\theta}^{\mathrm{BM}}(\psi^{-1}(\tilde Y))$, for $\tilde Y = \psi(Y) \in \R^{2d \times 2r}$ the real-valued representation of $Y$ through the isomorphism $\psi$. This new posterior can be further written as 
\[ \tilde \nu_{\lambda,\theta}^{\mathrm{BM}}(\tilde Y) = \exp\left(-\lambda \sum_{s=1}^{2^n} \sum_{a=1}^{3^n} \left( \hat p_{a,s} - \sqrt{2} \tr(\tilde \calP_s^a \cdot \tilde Y\tilde Y^\top) \right)^2 \right) \tilde \nu_{\theta}(\tilde Y), \]
where the additional factor $\sqrt{2}$ results from \eqref{eq:_in_pro}, and where $\tilde \nu_{\theta}(\tilde Y) := \nu_{\theta}(\psi^{-1}(\tilde Y))$. The negative log-density of the posterior is given by
\begin{align*} \label{eq:posterior_real}
 \tilde f_{\lambda,\theta}(\tilde Y) &= - \log(\tilde \nu_{\lambda,\theta}^{\mathrm{BM}}(\tilde Y)) \\
 &= \lambda  \sum_{s=1}^{2^n} \sum_{a=1}^{3^n} \left( \hat p_{a,s} - \sqrt{2} \tr(\tilde \calP_s^a \cdot \tilde Y\tilde Y^\top) \right)^2 + \frac{2d+r+2}{4} \log \det \left(\frac{\theta^2}{\sqrt{2}} I_{2d} + \sqrt{2} \tilde Y\tilde Y^\top\right) + \hat C, 
\end{align*}
with $\hat C = -\log(C_{\theta}) + (2d+r+2)d \log(2)/4$, where $C_\theta$ is the normalization constant of the prior \eqref{eq:prior}. The gradient of this function is obtained as:
\[ \nabla \tilde f_{\lambda,\theta}(\tilde Y) = -2 \sqrt{2} \lambda \sum_{s=1}^{2^n} \sum_{a=1}^{3^n}  (\hat p_{a,s} - \sqrt{2} \tr(\tilde \calP_a^s \cdot \tilde Y\tilde Y^\top)) (\tilde \calP_a^s+\tilde \calP_a^{s\top}) \tilde Y + \frac{2d+r+2}{\theta^2} \left( I_{2d} + \frac{2}{\theta^2}\tilde Y\tilde Y^\top\right)^{-1} \tilde Y, \]
where the derivative of the determinant was computed using \cite[Result 10.3.3 (1)]{Lutkepohl1996}. According to the Sherman-Morrison-Woodbury formula \cite[Sec B.10]{Higham2008}, the matrix inverse can be written as 
\[ \left( I_{2d} + \frac{2}{\theta^2} \tilde Y\tilde Y^\top\right)^{-1} = I_{2d} - \tilde Y \left(\frac{\theta^2}{2}I_{2r} + \tilde Y^\top \tilde Y\right)^{-1}\tilde Y^\top, \]
which is substantially cheaper to evaluate in the case $r \ll d$ (recalling that $r = 1$ for the recovery of a pure state).

\Cref{algo:sampler} summarizes the computation of our proposed \emph{Burer-Monteiro estimator} (BM-estimator in short). Note that this algorithm does not impose the unit-trace constraint on the estimator during sampling; instead, the trace is only normalized before termination. This is motivated by the observation that the trace of the density estimator typically remains close to one during the sampling process. 

\begin{algorithm}[H]
\begin{algorithmic}[1]
\State \textbf{Input:} rank estimate $r$, initial point $Y_0 \in \C^{d \times r}$, hyperparameters $\eta$, $\beta$, $\theta$, $\lambda$, $n_{\mathrm{burnin}}$.
\State Compute $\tilde Y_0 := \psi(Y_0)$
\For{$k \geq 1$ until the termination criterion is satisfied}
\State Sample the entries of $w_k^R, w_k^I \in \R^{d \times r}$ i.i.d. at random according to the standard Gaussian distribution, and let $w_k := w_k^R + i w_k^I$, and $\tilde w_k := \psi(w_k)$.  
\State $\tilde Y_k = \tilde Y_{k-1} - \eta \nabla \tilde f_{\lambda,\theta}(\tilde Y_{k-1}) + \frac{\sqrt{2 \eta}}{\beta} \tilde w_k $
\EndFor
\State Let $\hat \rho = \frac{1}{k-n_{\mathrm{burnin}}} \sum_{i = 1}^{k-n_\text{burnin}} \psi^{-1}(\tilde Y_{k-i+1}) \psi^{-1}(\tilde Y_{k-i+1})^* $ 
\State \textbf{Return:} the estimator $\frac{1}{\tr(\hat \rho)} \hat \rho$.
\end{algorithmic}
\caption{BM-estimator for quantum tomography}
\label{algo:sampler}
\end{algorithm}

We next compare our proposed Langevin sampling algorithm with the prob-estimator proposed in \cite{Mai2017} on a collection of synthetic data.

\paragraph{Data generation:} Similarly to \cite{Mai2017}, we assume the target matrix $\rho$ to be of one of the following forms.
\begin{enumerate}[]
    \item A pure state density $\rho_{\mathrm{rank \; 1}}^0 = v v^*$ for some $v \in \mathbb{C}^d$.
    \item A rank-2 density $\rho_{\mathrm{rank \; 2}}^0 = \frac{1}{2}v_1 v_1^* + \frac{1}{2}v_2 v_2^*$, where $v_1, v_2$ are two normalized orthogonal vectors in $\mathbb{C}^{d}$.  \label{rank2}
    \item An approximate rank-2 density $\rho_{\mathrm{approx \; rank \; 2}}^0 = w \rho_{\mathrm{rank \; 2}}^0 + (1 - w) I_d/d$, $w = 0.98$, where $\rho^0_{\mathrm{rank \; 2}}$ was generated as above.
    \item A maximal mixed state (rank $d$).
\end{enumerate}

We focus on the complete measurement setting, in which the number of experiments is given by $4^n$, corresponding for each qubit to the three Pauli observables and the identity operator. We depart from the per qubit measurement setting described in previous sections, and consider the more general setting where the system is measured as a whole: for experiment $a \in \{1, \dots, 4^n \}$, only a single scalar is available, instead of an $n$-dimensional vector as described in our theoretical results that follow the setting used in \cite{Mai2017}. The outcome of an experiment is thus simulated using a binomial random variable $b \sim B(m,p)$, where $p$ is the probability to observe the (system) state considered, quantified theoretically by Born's rule, and $m$ is the number of (virtual) replications of the experiment.

\paragraph{Algorithms considered:} We compare the prob-estimator presented in \cite{Mai2017} with two variants of our proposed BM estimator. These two variants aim to account for the situation where the rank of the target density matrix is known or not. In the first case, our BM estimator is obtained by running \Cref{algo:sampler} with $r = 1, 2, 2$ for respectively a rank-1, rank-2, and approximately rank-2 target density matrix. In the second case, to reflect the situation where the rank of the target density matrix is unknown, we let $r = d$ for each type of target density.
Unless stated otherwise, we rely on the following hyperparameter values: $\eta = 10^{-5}$, $\beta = 10^3$, $\lambda = m/2$, $\theta \in \{0.1,100\}$, where the smallest value is chosen when the rank is unknown, to promote low-rankness in the prior density, and the second is chosen when $r$ is expected to be the target rank, in which case the prior density is closer to a uniform distribution across ranks.  Each algorithm is run for $10^4$ iterations. As suggested in \cite{Mai2017}, a burnin phase (by default of 2000 iterations) is used for each algorithm. We initialize our BM-estimators at some initial point $Y_0 = V D^{1/2}$, where $V \in \C^{d \times r}$ is a unitary matrix sampled uniformly at random from the Stiefel manifold (Haar distribution), and $D \in \R^{r \times r}$ is a diagonal matrix whose entries follow a Dirichlet distribution with parameter $1/r$. 

\paragraph{Error measure:} We use the distance (measured in terms of the Frobenius norm) between the estimator $\hat \rho$ obtained and the true density matrix $\rho^0$.

 \subsection{Comparison of our proposed BM-estimators with the prob-estimator}

\begin{figure*}[t!]
    \centering
    \begin{subfigure}[t]{0.45\textwidth}
        \centering
        \includegraphics[scale = 0.5,trim = 5 10 0 0, clip]{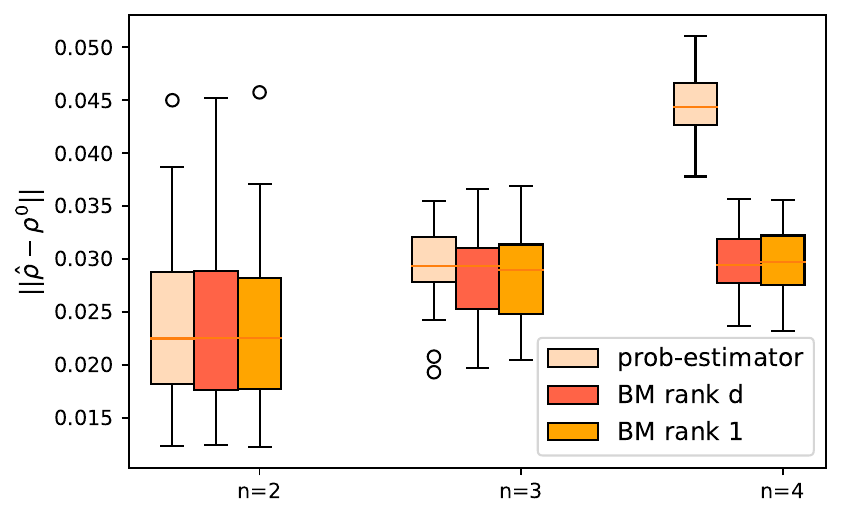}
        \caption{Rank-1 density matrix}
    \end{subfigure}%
    ~ 
    \begin{subfigure}[t]{0.45\textwidth}
        \centering
        \includegraphics[scale = 0.5,trim = 5 10 0 0, clip]{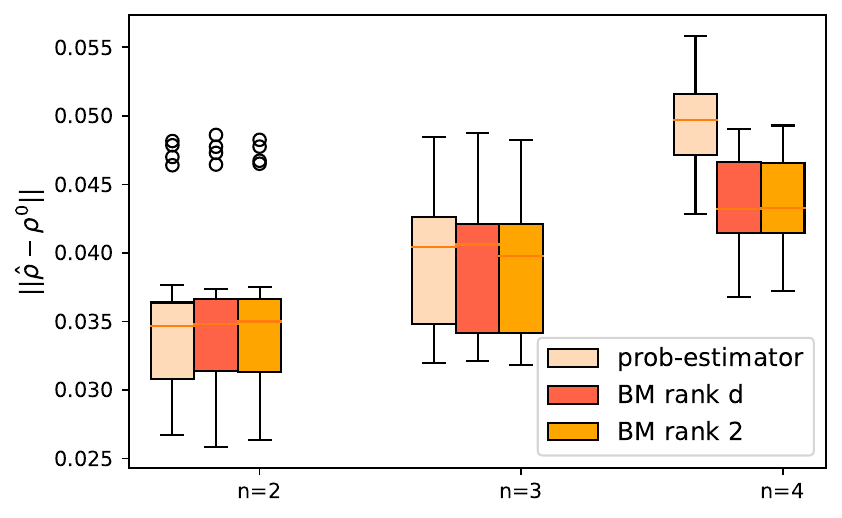}
        \caption{Rank-2 density matrix}
    \end{subfigure}
    ~
    \begin{subfigure}[t]{0.45\textwidth}
        \centering
        \includegraphics[scale = 0.5,trim = 5 10 0 0, clip]{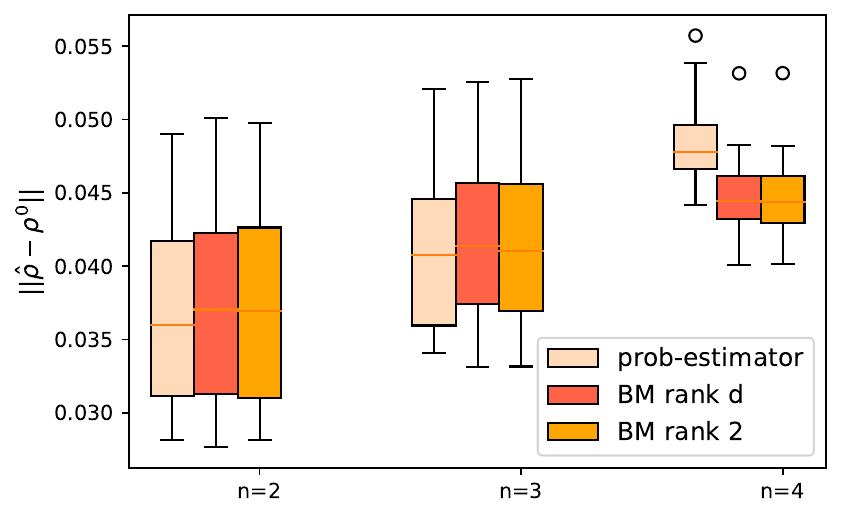}
        \caption{Approx. rank-2 density matrix}
    \end{subfigure}
    ~
    \begin{subfigure}[t]{0.45\textwidth}
    \centering
    \includegraphics[scale = 0.5,trim = 5 10 0 0, clip]{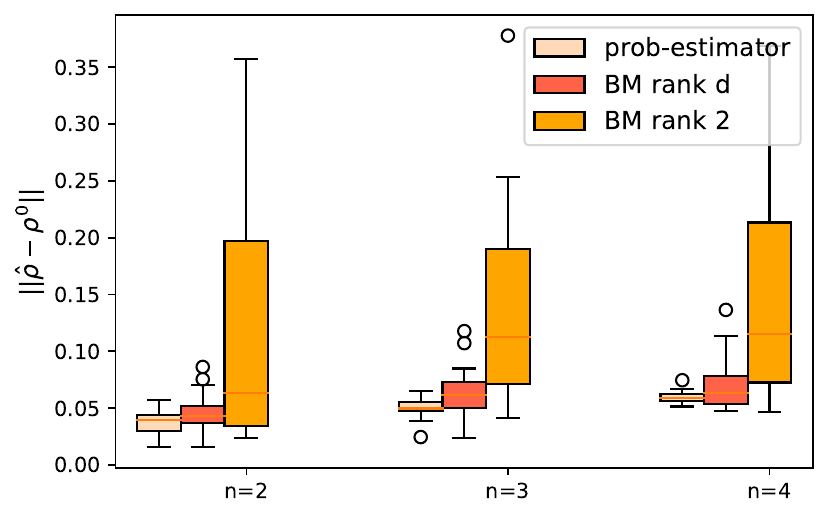}
    \caption{Rank-$d$ density matrix}
    \end{subfigure}
    \caption{Comparison of the final accuracies of the different estimators considered.}
    \label{ref:boxplots}
\end{figure*}

\Cref{ref:boxplots} compares two variants of our proposed BM-estimator with the prob-estimator derived in \cite{Mai2017}, in terms of accuracy, for different problem dimensions and low-rankness of the target density matrix. It shows that the final accuracy of the prob-estimator proposed in \cite{Mai2017} is in general comparable to the final accuracy of our BM estimators. For large matrices ($n = 4$), our proposed estimators substantially outperform the prob-estimator, both when the rank of the target matrix is known and when it is not. The fact that our BM estimators achieve the same accuracy regardless of the knowledge of the rank of the target density matrix supports the ability of our proposed prior to promote low-rankness. As a sanity check, we also displayed the results obtained by our BM estimator with $r$ significantly lower than the rank of the target density matrix, resulting as expected in a low accuracy; see the bottom-right plot in \Cref{ref:boxplots}, where a BM estimator with $r = 2$ was computed for a full-rank target density matrix. 

Beyond final accuracy, we compare in \Cref{fig:typical_run} the convergence of the Markov chains generated by each algorithm:  we display in \Cref{fig:typical_run} the evolution of the distance of each estimator to the target density matrix during a typical run of the algorithm (we only display the 2000 steps of each algorithm, with a burnin phase duration chosen here to be 800). This figure shows that the Markov chains associated with our BM-estimators converge much faster than the one associated with the prob-estimator proposed in \cite{Mai2017}. As indicated by \Cref{tab:comp_time}, the computational cost of generating a new sample of the Markov chain is also highly variable across the algorithms. For small dimensions ($n = 2$ or $n = 3$), the BM approach is faster than the prob-estimator of \cite{Mai2017}. On the other hand, for higher-dimensional problems, the BM approach with no rank information (i.e., $r = d$) becomes more costly than the prob-estimator; arguably due to the need to compute the gradient of the logarithm of the prior that involves a matrix inverse. This issue is mitigated for small values of $r$, by exploiting the Sherman-Morrison-Woodbury formula described in the previous section, that replaces the inverse of a $d \times d$ matrix of a rank-$r$ perturbation of the identity by the inverse of a $r \times r$ matrix. The use of this formula substantially ensures the scalability of the BM-estimators with problem dimension, resulting for $n = 5$ in a computational cost to generate a sample of the Markov chain that is about one order of magnitude below the one of the prob-estimator.

\begin{figure*}[t!]
    \centering
    \begin{subfigure}[t]{0.45\textwidth}
        \centering
        \includegraphics[scale = 0.5,trim = 5 10 0 0, clip]{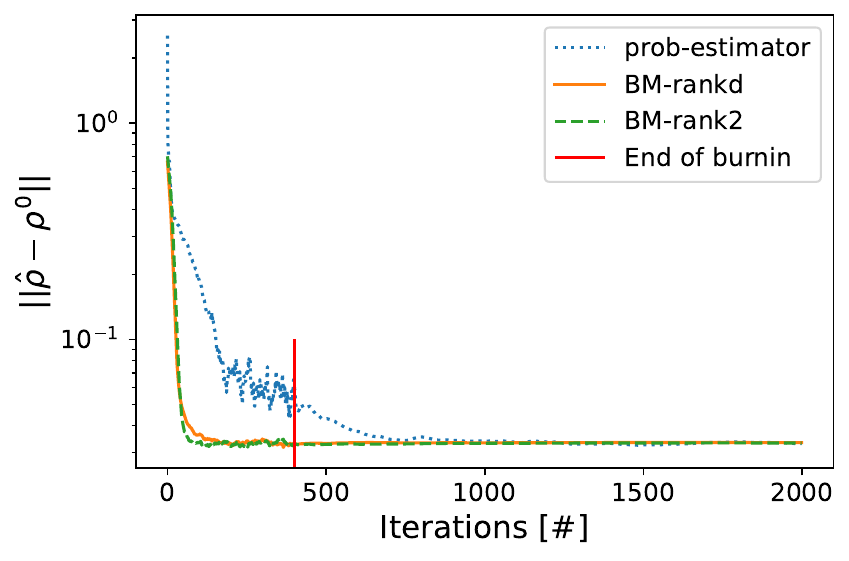}
        \caption{$n = 2$}
    \end{subfigure}%
    ~ 
    \begin{subfigure}[t]{0.45\textwidth}
        \centering
        \includegraphics[scale = 0.5,trim = 5 10 0 0, clip]{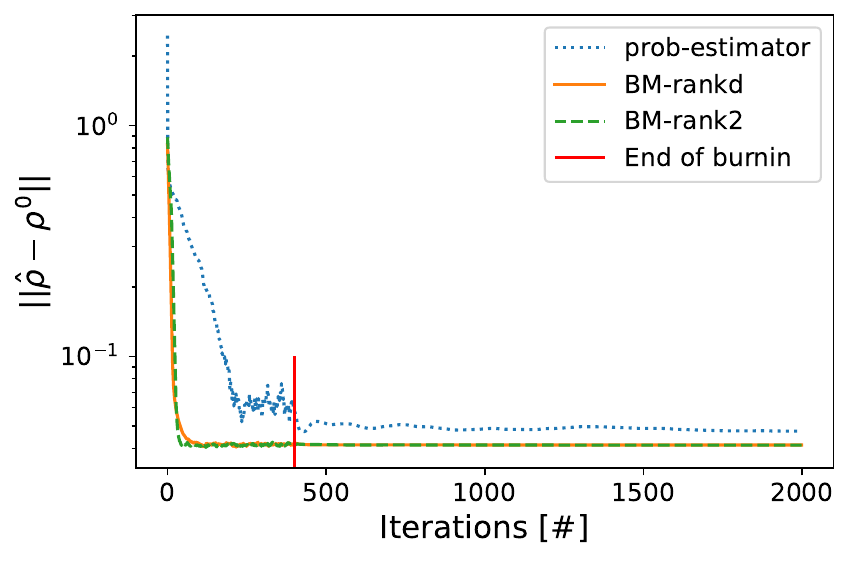}
        \caption{$n = 3$}
    \end{subfigure}
    ~ 
    \begin{subfigure}[t]{0.45\textwidth}
        \centering
        \includegraphics[scale = 0.5,trim = 5 10 0 0, clip]{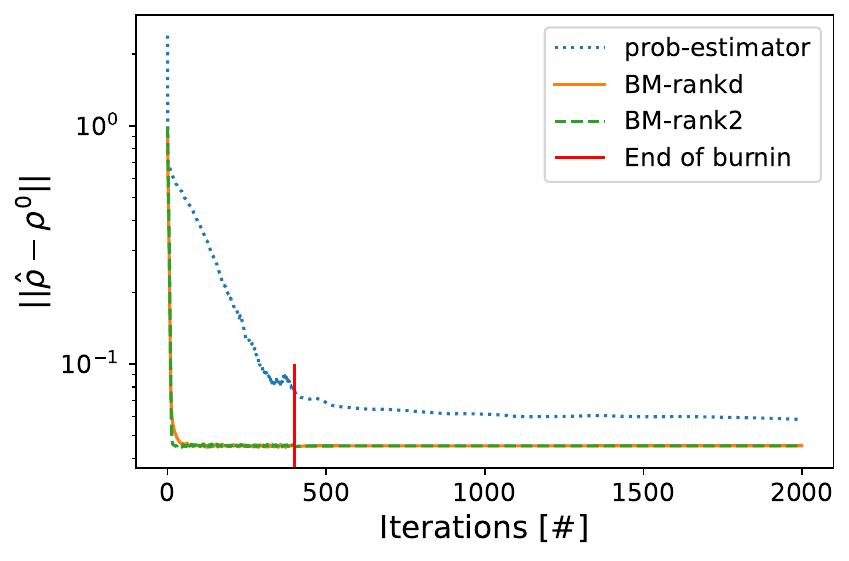}
        \caption{$n = 4$}
    \end{subfigure}
    ~
    \begin{subfigure}[t]{0.45\textwidth}
        \centering
        \includegraphics[scale = 0.5,trim = 5 10 0 0, clip]{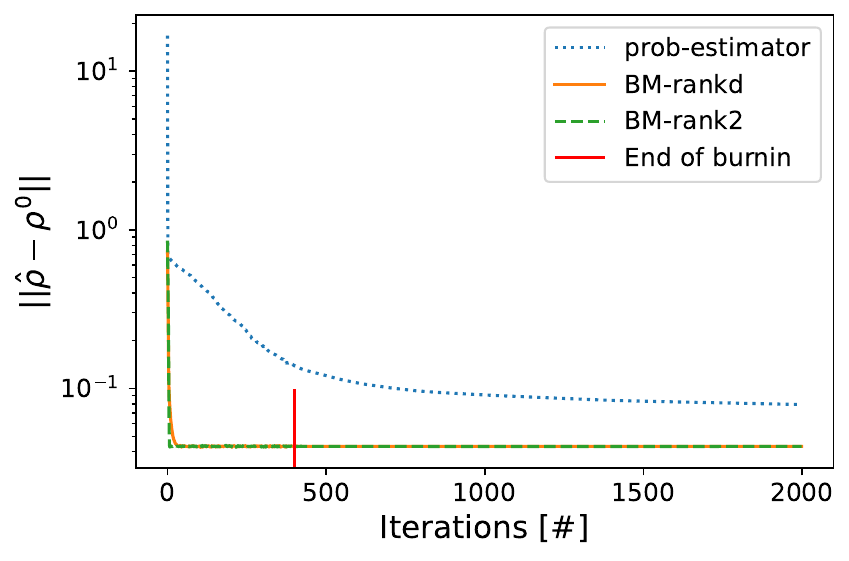}
        \caption{$n = 5$}
    \end{subfigure}
    \caption{Convergence of the Markov chain generated by each sampling algorithm, for a rank-2 target density matrix of different dimensions. }
    \label{fig:typical_run}
\end{figure*}

\begin{table}[t!]
    \centering
    \begin{tabular}{c|c|c|c}
         & prob-estimator &  BM-rank 2 & BM-rank $d$\\ 
         \hline
        $n=2$ & $5.7$ & $1.3$ & $1.6$ \\
        $n=3$ & $ 12 $  & $1.7$ & $3$\\
        $n=4$ & $124.8$ & $37$ & $198.7$ \\
        $n=5$ & $ 2040$ & $319$ & $3593$\\
    \end{tabular}
    \caption{Computation time per iteration ($\times 10^{-4}\;\mathrm{s}$)}
    \label{tab:comp_time}
\end{table}

Finally, \Cref{fig:sampler} displays the performance of our BM-estimator when varying the number of replications $m$ of each experiment, for $n = 3$, target density matrix $\rho^0_{\mathrm{rank 2}}$, $\lambda = m/2$, $\theta = 10^2$. Each point on \Cref{fig:sampler} was obtained as an average of the result of \Cref{algo:sampler} for 10 random seeds. We see in \Cref{fig:sampler} that the error decreases with the sampling size, with a rate $1/m$, in accordance with the theoretical analysis of the last section; indeed, the orange line in \Cref{fig:sampler} is the line obtained by linear regression of all points of the curve (except for the two first), and has a slope equal to -0.99. 

\begin{figure}
    \centering
    \includegraphics[scale = 0.7]{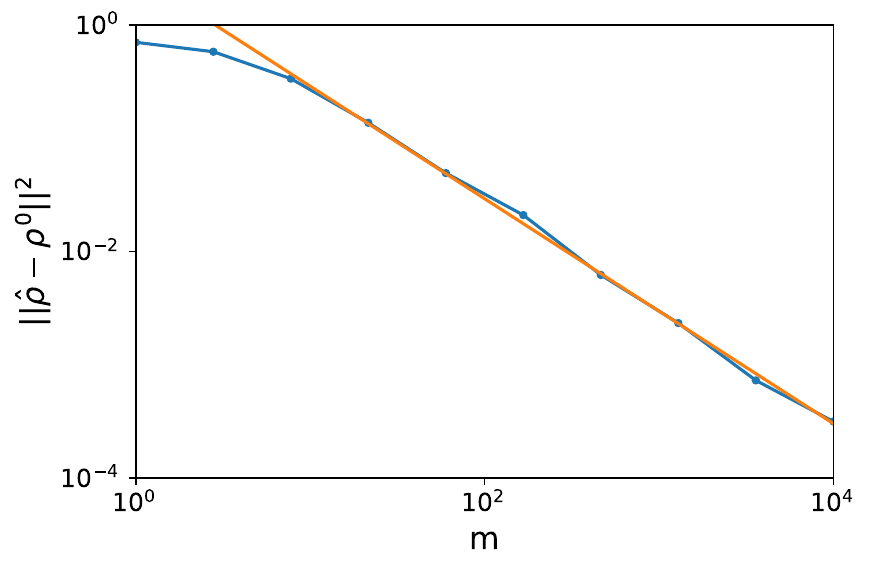}
    \caption{Relative error achieved by \Cref{algo:sampler} when increasing $m$, the number of repetition of each experiment (blue). The orange line is the result of a linear regression of all points on the curve (omitting the first two). }
    \label{fig:sampler}
\end{figure}

\section{Conclusions}
We propose a new estimator for quantum tomography that is efficiently computed using Langevin sampling, exploiting the differentiability of the associated log-posterior density. Our estimator relies on the Burer-Monteiro factorization of the density matrices, which naturally allows the incorporation of prior rank knowledge. When the rank of the target density is known (e.g., one is measuring a pure state), this factorization drastically improves the scalability of our proposed estimator compared to alternative estimators in the literature; a feature that is especially appealing given the exponential increase of problem dimension with the number of qubits in the system. When the rank of the target density matrix is unknown but only an upper bound is available, this upper bound can be used to reduce the computational burden per iteration of our Langevin sampling algorithm. Meanwhile, our proposed prior density, which is a direct extension to the complex setting of the one proposed in \cite{Dalalyan2020}, promotes further the low-rankness of the iterates. 
Note that, as the posterior is not log-concave due to the choice of the prior, we have no theoretical guarantees on the convergence of our Langevin sampler. However, our numerical experiments show that our Langevin sampler performs well in practice, in particular when the rank of the density matrix is known to be small, corresponding to low-dimensional Burer-Monteiro factors, and we also observe a substantial improvement in scalability compared to the approach proposed in \cite{Mai2017}. A similar prior which leads to a posterior density that is not log-concave was also used in \cite{Dalalyan2020} in a different setting, where good numerical performance was also demonstrated. 

\section*{Acknowledgments}

This work was funded in part by the UK Government's Department for Science, Innovation and Technology (DSIT) through the UK's National Measurement System (NMS) programme and the National Quantum Technologies Programme.  This work was initiated and partly undertaken while E. Massart was a postdoctoral researcher at the Mathematical Institute, University of Oxford, funded by the National Physical Laboratory. Part of the implementations used for the numerical results presented in this paper were done by Danila Mokeev as part of his MSc thesis at UCLouvain. 

\bibliographystyle{alpha}
{\footnotesize
	\bibliography{biblio_proj_Lang}}

@article{Alquier2013,
    author = {Pierre Alquier and Cristina Butucea and Mohamed Hebiri and Katia Meziani and Tomoyuki Morimae},
    title = {Rank-penalized estimation of a quantum system},
    journal = {Physical Review A} ,
    year = 2013,
    volume = {88},
    number = {032113},
}

@article{Neal2011a,
    author = {Radford M. Neal},
    title = {{MCMC using Hamiltonian dynamics}},
    journal = {Handbook of Markov Chain Monte Carlo} ,
    year = {2011}
}

@Article{Blume2010,
  author    = {R. Blume-Kohout},
  title     = {Optimal, reliable estimation of quantum states},
  journal   = {New Journal of Physics},
  year      = {2010},
  volume    = {12},
  number    = {043034},
  doi       = {10.1088/1367-2630/12/4/043034},
}

@Article{Burer2003,
  author    = {S. Burer and R.D.C. Monteiro},
  title     = {A nonlinear programming algorithm for solving semidefinite programs via low-rank factorization},
  journal   = {Mathematical Programming},
  year      = {2003},
  volume    = {95},
  number    = {2},
  pages     = {329--357},
  doi       = {10.1007/s10107-002-0352-8},
}

@Article{Buzek1998,
  author    = {Bu\u{z}ek, V. and Derka, R. and Adam, G. and Knight, P.L.},
  journal   = {Annals of Physics},
  title     = {Reconstruction of Quantum States of Spin Systems: From Quantum {Bayesian} Inference to Quantum Tomography},
  year      = {1998},
  issn      = {0003-4916},
  number    = {2},
  pages     = {454--496},
  volume    = {266},
  doi       = {10.1006/aphy.1998.5802},
  publisher = {Elsevier BV},
}

@book{Catoni2007,
    author = {Olivier Catoni},
    title = {Pac-Bayesian Supervised Classification: The Thermodynamics of Statistical Learning},
    publisher = {IMS Lecture Notes Monogr},
    number = {56},
    year = {2007}
}

@Article{Dalalyan2012,
  author    = {Dalalyan, A.S. and Tsybakov, A.B.},
  journal   = {Journal of Computer and System Sciences},
  title     = {Sparse regression learning by aggregation and {Langevin Monte-Carlo}},
  year      = {2012},
  issn      = {0022-0000},
  number    = {5},
  pages     = {1423--1443},
  volume    = {78},
  doi       = {10.1016/j.jcss.2011.12.023},
  publisher = {Elsevier BV},
}

@Article{Dalalyan2020,
  author    = {Arnak S. Dalalyan},
  title     = {Exponential weights in multivariate regression and a low-rankness favoring prior},
  journal   = {Ann. Inst. H. Poincaré Probab. Statist.},
  year      = {2020},
  volume    = {56}, 
  number    = {2},
  pages     = {1465-1483},
  doi       = {10.1214/19-AIHP1010},
}

@Article{Granade2016,
  author    = {C. Granade and J. Combes and D. G. Cory},
  title     = {Practical {Bayesian} tomography},
  journal   = {New Journal of Physics},
  year      = {2016},
  volume    = {18},
  number    = {133024},
  pages     = {259--296},
  doi       = {10.1088/1367-2630/18/3/033024},
}

@Article{Gross2010,
  author    = {David Gross and Yi-Kai Liu and Steven T. Flammia and Stephen Becker and Jens Eisert},
  title     = {Quantum State Tomography via Compressed Sensing},
  year      = {2010},
  volume    = {105},
  number    = {15},
  doi       = {10.1103/physrevlett.105.150401},
  journal = {Phys. Rev. Lett.},
  publisher = {American Physical Society ({APS})},
}

@misc{Guedj2019,
    author = {Benjamin Guedj},
    title = {A primer on {PAC Bayesian} learning},
    howpublished = {arXiv:1901.05353},
    year         = {2019},
}

@book{Lutkepohl1996,
    author = {H. Lutkepohl},
    title = {Handbook of matrices},
    publisher = {John Wiley and Sons},
    year = {1996},
}

@Article{Guta2012,
  author    = {Guţă, Mădălin and Kypraios, Theodore and Dryden, Ian},
  journal   = {New Journal of Physics},
  title     = {Rank-based model selection for multiple ions quantum tomography},
  year      = {2012},
  issn      = {1367-2630},
  number    = {10},
  pages     = {105002},
  volume    = {14},
  doi       = {10.1088/1367-2630/14/10/105002},
  publisher = {IOP Publishing},
}

@Article{Guta2020,
  author    = {M. Gu{\c{t}}{\u{a}} and J. Kahn and R. Kueng and J. A. Tropp},
  title     = {Fast state tomography with optimal error bounds},
  journal   = {Journal of Physics A: Mathematical and Theoretical},
  year      = {2020},
  volume    = {53},
  number    = {20},
  pages     = {204001},
  doi       = {10.1088/1751-8121/ab8111},
  publisher = {{IOP} Publishing},
}

@book{Higham2008,
    author = {Nicholas J. Higham},
    title = {Functions of Matrices},
    publisher = {SIAM},
    year = {2008}
}

@book{HornJohnson2013,
    author = {Roger A. Horn and Charles R. Johnson},
    title = {Matrix Analysis (Second Edition)},
    publisher = {Cambridge University Press},
    year = {2013}
}

@Article{Huszar2012,
  author    = {F. Huszar and N. M. T. Houlsby},
  title     = {Adaptive {Bayesian} quantum tomography},
  journal   = {Physical Review A},
  year      = {2012},
  volume    = {85},
  number    = {052120},
  doi       = {10.1103/PhysRevA.85.052120},
}

@InBook{Hradil2004,
  author    = {Hradil, Zden{\u{e}}k and {\u{R}}eh\'a{\u{c}}ek, Jaroslav and Fiur\'a{\u{s}}ek, Jarom\'ir and Je{\u{z}}ek, Miroslav},
  pages     = {59--112},
  publisher = {Springer Berlin Heidelberg},
  title     = {3 Maximum-Likelihood Methods in Quantum Mechanics},
  year      = {2004},
  isbn      = {9783540444817},
  booktitle = {Quantum State Estimation},
  doi       = {10.1007/978-3-540-44481-7_3},
  issn      = {1616-6361},
}

@Article{Jones1991,
  author    = {Jones, K.R.W.},
  journal   = {Annals of Physics},
  title     = {Principles of quantum inference},
  year      = {1991},
  issn      = {0003-4916},
  number    = {1},
  pages     = {140--170},
  volume    = {207},
  doi       = {10.1016/0003-4916(91)90182-8},
  publisher = {Elsevier BV},
}

@Article{Kravtsov2013,
  author    = {Kravtsov, K. S. and Straupe, S. S. and Radchenko, I. V. and Houlsby, N. M. T. and Huszár, F. and Kulik, S. P.},
  journal   = {Physical Review A},
  title     = {Experimental adaptive {Bayesian} tomography},
  year      = {2013},
  issn      = {1094-1622},
  number    = {6},
  pages     = {062122},
  volume    = {87},
  doi       = {10.1103/physreva.87.062122},
  publisher = {American Physical Society (APS)},
}

@Article{Kueng2015,
  author    = {Kueng, Richard and Ferrie, Christopher},
  journal   = {New Journal of Physics},
  title     = {Near-optimal quantum tomography: estimators and bounds},
  year      = {2015},
  issn      = {1367-2630},
  number    = {12},
  pages     = {123013},
  volume    = {17},
  doi       = {10.1088/1367-2630/17/12/123013},
  publisher = {IOP Publishing},
}

@Article{Lukens2020,
  author    = {Joseph M. Lukens and Kody J. H. Law and Ajay Jasra and Pavel Lougovski},
  title     = {A practical and efficient approach for {Bayesian} quantum state estimation},
  journal   = {New Journal of Physics},
  year      = {2020},
  volume    = {22},
  number    = {6},
  pages     = {063038},
  doi       = {10.1088/1367-2630/ab8efa},
  publisher = {{IOP} Publishing},
}

@Article{Mai2017,
  author    = {The Tien Mai and Pierre Alquier},
  title     = {{Pseudo-Bayesian} quantum tomography with rank-adaptation},
  journal   = {Journal of Statistical Planning and Inference},
  year      = {2017},
  volume    = {184},
  pages     = {62--76},
  doi       = {10.1016/j.jspi.2016.11.003},
  publisher = {Elsevier {BV}},
}

@Misc{Mai2021,
  author       = {The Tien Mai},
  title        = {{An efficient adaptive MCMC algorithm for Pseudo‑Bayesian quantum tomography}},
  journal = {Computational Statistics},
  year = {2022},
}

@Article{Ollila2012,
  author    = {Esa Ollila and David E. Tyler and Visa Koivunen and H. Vincent Poor},
  title     = {Complex Elliptically Symmetric Distributions: Survey, New Results and Applications},
  journal   = {{IEEE} Transactions on Signal Processing},
  year      = {2012},
  volume    = {60},
  number    = {11},
  pages     = {5597--5625},
  doi       = {10.1109/TSP.2012.2212433},
}

@Article{Quek2021,
  author    = {Quek, Yihui and Fort, Stanislav and Ng, Hui Khoon},
  journal   = {npj Quantum Information},
  title     = {Adaptive quantum state tomography with neural networks},
  year      = {2021},
  issn      = {2056-6387},
  number    = {1},
  volume    = {7},
  doi       = {10.1038/s41534-021-00436-9},
  publisher = {Springer Science and Business Media LLC},
}

@book{Nielsen_Chuang_2010, 
place={Cambridge}, 
title={Quantum Computation and Quantum Information: 10th Anniversary Edition}, 
publisher={Cambridge University Press}, 
author={Nielsen, Michael A. and Chuang, Isaac L.}, 
year={2010}}

@misc{lall2025reviewcollectionmetricsbenchmarks,
      title={A Review and Collection of Metrics and Benchmarks for Quantum Computers: definitions, methodologies and software}, 
      author={Deep Lall and Abhishek Agarwal and Weixi Zhang and Lachlan Lindoy and Tobias Lindström and Stephanie Webster and Simon Hall and Nicholas Chancellor and Petros Wallden and Raul Garcia-Patron and Elham Kashefi and Viv Kendon and Jonathan Pritchard and Alessandro Rossi and Animesh Datta and Theodoros Kapourniotis and Konstantinos Georgopoulos and Ivan Rungger},
      year={2025},
      eprint={2502.06717},
      archivePrefix={arXiv},
      primaryClass={quant-ph},
      url={https://arxiv.org/abs/2502.06717}, 
}

@misc{preskill1998lecture,
  title={Lecture notes for physics 229: Quantum information and computation},
  author={Preskill, John},
  publisher = {California Institute of Technology},
  volume={16},
  number={1},
  pages={1--8},
  year={1998},
  url = {https://web.gps.caltech.edu/~rls/book.pdf}

}

@article{bluvstein2024logical,
  title        = {Logical Quantum Processor Based on Reconfigurable Atom Arrays},
  author       = {Bluvstein, Dolev and Evered, Simon J. and Geim, Alexandra A. and Li, Sophie H. and Zhou, Hengyun and Manovitz, Tom and Ebadi, Sepehr and Cain, Madelyn and Kalinowski, Marcin and Hangleiter, Dominik and Bonilla Ataides, J. Pablo and Maskara, Nishad and Cong, Iris and Gao, Xun and Sales Rodriguez, Pedro and Karolyshyn, Thomas and Semeghini, Giulia and Gullans, Michael J. and Greiner, Markus and Vuletić, Vladan and Lukin, Mikhail D.},
  date         = {2024},
  year         = {2024},
  journal = {Nature},
  volume       = {626},
  number       = {7997},
  pages        = {58},
  publisher    = {Nature Publishing Group UK London},
  url          = {https://doi.org/10.1038/s41586-023-06927-3}
}

@article{acharya2024quantum,
  title   = {Quantum error correction below the surface code threshold},
  author  = {{Google Quantum AI and Collaborators}},
  journal={Nature},
  volume={638},
  number={8052},
  pages={920},
  year={2024},
  url={https://doi.org/10.1038/s41586-024-08449-y}
}

@article{vogel1989,
  title = {Determination of quasiprobability distributions in terms of probability distributions for the rotated quadrature phase},
  author = {Vogel, K. and Risken, H.},
  journal = {Physical Review A},
  volume = {40},
  issue = {5},
  pages = {2847--2849},
  numpages = {0},
  year = {1989},
  publisher = {American Physical Society},
  doi = {10.1103/PhysRevA.40.2847},
  url = {https://link.aps.org/doi/10.1103/PhysRevA.40.2847}
}

@article{proctor2025benchmarking,
  title={Benchmarking quantum computers},
  author={Proctor, Timothy and Young, Kevin and Baczewski, Andrew D and Blume-Kohout, Robin},
  journal={Nature Reviews Physics},
  pages={1--14},
  year={2025},
  publisher={Nature Publishing Group UK London}
}

@article{farooq2022robust,
  title={Robust quantum state tomography method for quantum sensing},
  author={Farooq, Ahmad and Khalid, Uman and ur Rehman, Junaid and Shin, Hyundong},
  journal={Sensors},
  volume={22},
  number={7},
  pages={2669},
  year={2022},
  publisher={MDPI},
  url={https://doi.org/10.3390/s22072669}
}

@article{Bouchard2019quantumprocess,
  doi = {10.22331/q-2019-05-06-138},
  url = {https://doi.org/10.22331/q-2019-05-06-138},
  title = {Quantum process tomography of a high-dimensional quantum communication channel},
  author = {Bouchard, Fr{\'{e}}d{\'{e}}ric and Hufnagel, Felix and Koutn{\'{y}}, Dominik and Abbas, Aazad and Sit, Alicia and Heshami, Khabat and Fickler, Robert and Karimi, Ebrahim},
  journal = {{Quantum}},
  issn = {2521-327X},
  publisher = {{Verein zur F{\"{o}}rderung des Open Access Publizierens in den Quantenwissenschaften}},
  volume = {3},
  pages = {138},
  year = {2019}
}

@article{gebhart2023learning,
  title={Learning quantum systems},
  author={Gebhart, Valentin and Santagati, Raffaele and Gentile, Antonio Andrea and Gauger, Erik M. and Craig, David and Ares, Natalia and Banchi, Leonardo and Marquardt, Florian and Pezze, Luca and Bonato, Cristian},
  journal={Nature Reviews Physics},
  volume={5},
  number={3},
  pages={141--156},
  year={2023},
  publisher={Nature Publishing Group UK London},
  url={https://doi.org/10.1038/s42254-022-00552-1}
}

@article{Suess_2017,
doi = {10.1088/1367-2630/aa7ce9},
url = {https://dx.doi.org/10.1088/1367-2630/aa7ce9},
year = {2017},
publisher = {IOP Publishing},
volume = {19},
number = {9},
pages = {093013},
author = {Suess, Daniel and Rudnicki, Łukasz and maciel, Thiago O and Gross, David},
title = {Error regions in quantum state tomography: computational complexity caused by geometry of quantum states},
journal = {New Journal of Physics},
}

@article{agarwal2024modelling,
  title     = {Modelling non-Markovian noise in driven superconducting qubits},
  author    = {Agarwal, Abhishek and Lindoy, Lachlan P and Lall, Deep and Jamet, Fran{\c{c}}ois and Rungger, Ivan},
  journal   = {Quantum Sci. Technol.},
  volume    = {9},
  number    = {3},
  pages     = {035017},
  year      = {2024},
  publisher = {IOP Publishing},
  url       = {https://doi.org/10.1088/2058-9565/ad3d7e}
}

@article{agarwal2025fasttracking,
      title={Fast-tracking and disentangling of qubit noise fluctuations using minimal-data averaging and hierarchical discrete fluctuation auto-segmentation}, 
      author={Abhishek Agarwal and Lachlan P. Lindoy and Deep Lall and Sebastian E. de Graaf and Tobias Lindström and Ivan Rungger},
      year={2025},
      journal={arXiv:2505.23622},
      archivePrefix={arXiv},
      primaryClass={quant-ph},
      url={https://arxiv.org/abs/2505.23622}, 
}

\section{Appendix}

\subsection{Proof of Lemma 1}
 \begin{proof}
 This proof is a direct extension to the complex setting of the one given in \cite[Lemma 1]{Dalalyan2020}. For any bounded and measurable function $h : \C^d \to \R$, there holds
    \begin{align*}
        \int_{\C^{d \times r}} h(y_1) \nu_\theta(Y) \mathrm{d}Y &= C_\theta \int_{\C^{d \times r}} \frac{h(y_1)}{\det(\theta^2 I_d + YY^*)^{(2d+r+2)/2}} \mathrm{d}Y \\
        &= \tilde C_{\theta} \int_{\C^{d \times r}} \frac{h(\theta z_1)}{\det(I_d + ZZ^*)^{(2d+r+2)/2}} \mathrm{d}Z,
    \end{align*} 
    where $y_1 \in \C^d$ is the first column of $Y$, and where, for some $\tilde C_{\theta}$, the second equality follows from applying the change of variable $Z := Y/\theta$.
    We then make a second change of variable and write $x$ for the first column of $Z$, and  $Z = [x, Z_{2:r}] = [x, (I_d+xx^*)^{1/2} \bar Z_{2:r}]$.
    This yields 
    \begin{align*}
        \mathrm{d}Z = \mathrm{d}x \mathrm{d}Z_{2:r} &= \mathrm{d}x \det(I_d+xx^*)^{(r-1)/2} \mathrm{d} \bar Z_{2:r} \\
        &= (1+ \|x \|^2)^{(r-1)/2} \mathrm{d}x \mathrm{d} \bar Z_{2:r}
    \end{align*}  
    and 
    \begin{align*}
    \det(I_d+ZZ^*) &= \det(I_d + xx^* + Z_{2:r} Z_{2:r}^*) \\
    &= \det (I_d + xx^* + (I_d + xx^*)^{1/2} \bar Z_{2:r} \bar Z_{2:r}^* (I_d + xx^*)^{1/2}) \\
    &= \det ((I_d + xx^*)^{1/2} (I_d + \bar Z_{2:r} \bar Z_{2:r}^*) (I_d + xx^*)^{1/2}) \\
    &= \det (I_d + xx^*) \det (I_d + \bar Z_{2:r} \bar Z_{2:r}^*) \\
    &= (1 + \| x\|^2)  \det (I_d + \bar Z_{2:r} \bar Z_{2:r}^*).
    \end{align*}
    We then get 
    \begin{align*}
        \int_{\C^{d \times r}} h(y_1) \nu_\theta(Y) \rd dY &= \tilde C_{\theta}\int_{\C^{d \times r}}  h( \theta x) (1+ \|x\|^2)^{(r-1)/2} [(1 + \| x\|^2) \det (I_d + \bar Z_{2:r} \bar Z_{2:r}^*)]^{-(2d+r+2)/2} \rd dx  \rd \bar Z_{2:r}  \\
        &= \tilde C_{\theta}\int_{\C^{d \times r}} h( \theta x) (1+ \|x \|^2)^{-(2d+3)/2}  \det (I_d + \bar Z_{2:r} \bar Z_{2:r}^*)^{-(2d+r+2)/2} \rd x  \rd \bar Z_{2:r}.  
    \end{align*} 
Writing the normalization constant as
\[  \tilde{C}_\theta = \left(\int_{\C^{d \times r}} (1+ \|x \|^2)^{-(2d+3)/2}  \det (I_d + \bar Z_{2:r} \bar Z_{2:r}^*)^{-(2d+r+2)/2} \rd x  \rd \bar Z_{2:r}\right)^{-1} \]
gives 
 \begin{align*}
        \int_{\C^{d \times r}} h(y_1) \nu_\theta(Y) \rd Y &= \frac{\int_{\C^{d}} h( \theta x) (1+ \|x\|^2)^{-(2d+3)/2} \rd x}{\int_{\C^{d}} (1+ \|x \|^2)^{-(2d+3)/2} \rd x} \\
        &= \frac{\int_{\C^{d}} h(\sqrt{2/3} \theta y) (1+ 2\|y\|^2/3)^{-(2d+3)/2} \rd y}{\int_{\C^{d}} (1+ 2\|y \|^2/3)^{-(2d+3)/2} \rd y}, 
\end{align*}
where we recognize in the last expression the complex multivariate $t_3$-distribution with location $0$ and scale matrix $\Sigma = I_d$ \cite{Ollila2012}, whose p.d.f $\mu$ is proportional to $\mu(y) \propto (1 + 2 \|y\|^2/3)^{-(2d+3)/2} \rd y$. 

As the covariance matrix of the complex multivariate $t_p$-distribution is $\frac{p}{p-2} I_d$, it follows that 
\[ \int_{\C^{d \times r}} \|Y \|_\F^2 \nu_\theta(Y) \rd Y = \int_{\C^{d \times r}} \left(\sum_{i = 1}^r \|y_i \|^2 \right) \nu_\theta(Y) \rd Y = \sum_{i = 1}^r \int_{\C^{d \times r}} \|y_i \|^2 \nu_\theta(Y) \rd Y = \frac{2r}{3} \theta^2 \int_{\C^d} \|y\|^2 \mu(y) \rd y, \]
so that 
\[ \int_{\C^{d \times r}} \|Y \|_\F^2 \nu_\theta(Y) \rd Y = 2rd \theta^2.\]
    \end{proof}

\subsection{Proof of Theorem 1}

 We write $P^0$, $P(\rho)$ and $\hat P \in [0,1]^{3^n \times 2^n}$ for the matrices defined as:
\[ [P^0]_{a,s} = p^0_{a,s} = \tr(\calP_s^a \cdot Y^0 Y^{0*}), \]
the matrix that quantifies the probability of experiment $a$ to result in outcome vector $o(s)$, for the system associated to the target density $\rho^0 = Y^0 Y^{0*}$  (see the notation introduced in \Cref{sec:gen_framework}), 
\[ [P(YY^*)]_{a,s} = \tr(\calP_s^a \cdot Y Y^{*}), \]
the matrix that quantifies the probability of experiment $a$ to result in outcome $o(s)$, for the system associated to an arbitrary density density $\rho = Y Y^{*}$, and 
\[ [\hat P]_{a,s} = \hat p_{a,s}, \]
the matrix of empirical frequencies \eqref{eq:empirical_freq}, obtained by letting $m_a = m$ for all $a$, i.e., the same number of replications is used for each experiment $a \in \{1, \dots, 3^n\}$. With this notation, our likelihood \eqref{eq:likelihood} can be written as 
\[  L^{\mathrm{BM}}: \C^{d \times r} \to \R : Y  \mapsto  \| \hat P - P(YY^*)\|^2_\F, \]
and the associated posterior distribution \eqref{eq:BM_posterior} can be written as
 $\hat \nu_{\lambda,\theta}^{\mathrm{BM}}(Y) = e^{-f_{\lambda,\theta}(Y)}$, with, according to the prior definition \eqref{eq:prior},
\begin{equation} \label{eq:posterior-complex}
    f_{\lambda,\theta}(Y) = \lambda L^{\mathrm{BM}}(Y) + \frac{2d+r+2}{2} \log \det(\theta^2 I_d + YY^*) +  C,
\end{equation}
and $C = -\log C_{\theta}$, where $C_{\theta}$ is the spectral scaled Student's $t$-distribution normalizing constant. We first recall the following lemma. 
\begin{lemma}\cite[Lemma 3]{Mai2017} \label{Lem:MaiLem3}
    For $\lambda > 0$, there holds for all $Y \in \C^{d \times r}$
    \begin{align}
        \E \exp \left\{ \lambda ( \|P(YY^*) - \hat P \|^2_\F - \| P^0 - \hat P \|^2_\F) - \lambda \left(1 + \frac{\lambda}{m} \right) \|P^0 - P(YY^*) \|^2_\F \right\} \leq 1 \\
        \E \exp \left\{ \lambda \left(1 - \frac{\lambda}{m} \right) \|P^0 - P(YY^*) \|^2_\F  - \lambda (\|P(YY^*) - \hat P \|^2_\F - \| P^0 - \hat P \|^2_\F) \right\} \leq 1,
        \end{align}
        where the expectations are over the random variables $R_i^a$ in $\hat P$, i.e., over the results of the $m$ replications of each experiment.
\end{lemma}

The following lemma is a straightforward adaptation of \cite[Lemma 4]{Mai2017}. 

\begin{lemma} \label{lem:Catoni_application}
    For $\lambda > 0$ such that $\frac{\lambda}{m} < 1$, with probability $1-\epsilon$, $\epsilon \in ]0,1[$, there holds {\small
\begin{equation} \label{eq:final_step_proba}
    \int \|P^0 - P(YY^*) \|^2_\F \hat \nu_{\lambda_{\mathrm{new}},\theta}^{\mathrm{BM}}(Y) \rd Y \leq \inf_{\nu  \in \Pi(\C^{d \times r})} \frac{\left(1 + \frac{\lambda}{m} \right) \int \|P^0 - P(YY^*) \|^2_\F \nu(Y) \rd Y   +2 \frac{\log\left(\frac{2}{\epsilon}\right) + D_\KL(\nu, \nu_\theta)}{\lambda}}{1-\frac{\lambda}{m}},
\end{equation}}
    with $\Pi(\C^{d \times r})$ the set of probability measures on $\C^{d \times r}$ and $D_\KL(\nu, \nu_\theta)$ the KL-divergence between an arbitrary probability measure $\nu \in \Pi(\C^{d \times r})$ and our prior density $\nu_\theta$, and $\lambda_{\mathrm{new}} := \frac{\lambda}{2}(1+\frac{\lambda}{m})$. 
\end{lemma}

\begin{proof}
    Due to \Cref{Lem:MaiLem3}, there holds
        \begin{align*}
         \int \E \exp \left\{ \lambda ( \|P(YY^*) - \hat P \|^2_\F - \| P^0 - \hat P \|^2_\F) - \lambda \left(1 + \frac{\lambda}{m} \right) \|P^0 - P(YY^*) \|^2_\F \right\} \nu_\theta(Y) \rd Y \leq 1, \\
        \int \E \exp \left\{ \lambda \left(1 - \frac{\lambda}{m} \right) \|P^0 - P(YY^*) \|^2_\F  - \lambda (\|P(YY^*) - \hat P \|^2_\F - \| P^0 - \hat P \|^2_\F) \right\} \nu_\theta(Y) \rd Y \leq 1.
        \end{align*}
        Fubini's theorem then gives 
        \begin{align*}
          \E \int \exp \left\{ \lambda ( \|P(YY^*) - \hat P \|^2_\F - \| P^0 - \hat P \|^2_\F) - \lambda \left(1 + \frac{\lambda}{m} \right) \|P^0 - P(YY^*) \|^2_\F \right\} \nu_\theta(Y) \rd Y \leq 1 \\
         \E \int \exp \left\{ \lambda \left(1 - \frac{\lambda}{m} \right) \|P^0 - P(YY^*) \|^2_\F  - \lambda (\|P(YY^*) - \hat P \|^2_\F - \| P^0 - \hat P \|^2_\F) \right\} \nu_\theta(Y) \rd Y \leq 1.
        \end{align*}    
        Using \cite[Lemma 1.1.3]{Catoni2007}, there holds for any $\epsilon > 0$
        \begin{align*}
        \E \exp \sup_{\nu \in \Pi(\C^{d \times r})} \Bigg\{ &\lambda \left( \int \|P(YY^*) - \hat P \|^2_\F \nu(Y) \rd Y - \| P^0 - \hat P \|^2_\F \right) - \log\left(\frac{2}{\epsilon}\right) - D_\KL(\nu, \nu_\theta) \\
        &-\lambda \left(1 + \frac{\lambda}{m} \right) \int \|P^0 - P(YY^*) \|^2_\F \nu(Y) \rd Y \Bigg\} \leq \frac{\epsilon}{2} \\
        \E \exp \sup_{\nu \in \Pi(\C^{d \times r})} \Bigg\{ &\lambda \left(1 - \frac{\lambda}{m} \right) \int \|P^0 - P(YY^*) \|^2_\F \nu(Y) dY - \log\left(\frac{2}{\epsilon}\right) - D_\KL(\nu, \nu_\theta) \\
        &- \lambda \left( \int \|P(YY^*) - \hat P \|^2_\F \nu(Y) dY - \| P^0 - \hat P \|^2_\F \right) \Bigg\} \leq \frac{\epsilon}{2}.
        \end{align*}
        Now, using $\mathbbm{1}_{x \geq 0}(x) \leq \exp(x)$, one has 
        \begin{align*}
        \P \Bigg \{\sup_{\nu \in \Pi(\C^{d \times r})} \Bigg[ &\lambda \left( \int \|P(YY^*) - \hat P \|^2_\F \nu(Y) \rd Y - \| P^0 - \hat P \|^2_\F \right) - \log\left(\frac{2}{\epsilon}\right) - D_\KL(\nu, \nu_\theta) \\
        &-\lambda \left(1 + \frac{\lambda}{m} \right) \int \|P^0 - P(YY^*) \|^2_\F \nu(Y) \rd Y \Bigg] \geq 0 \Bigg\} \leq \frac{\epsilon}{2} \\
        \P \Bigg\{ \sup_{\nu  \in \Pi(\C^{d \times r})} \Bigg[ &\lambda \left(1 - \frac{\lambda}{m} \right) \int \|P^0 - P(YY^*) \|^2_\F \nu(Y) \rd Y - \log\left(\frac{2}{\epsilon}\right) - D_\KL(\nu, \nu_\theta) \\
        &- \lambda \left( \int \|P(YY^*) - \hat P \|^2_\F  \nu(Y) \rd Y - \| P^0 - \hat P \|^2_\F \right) \Bigg] \geq 0 \Bigg\} \leq \frac{\epsilon}{2}.
        \end{align*}
We deduce that for any density $\nu  \in \Pi(\C^{d \times r})$, {\small
            \begin{align*}
        \P \Bigg \{ \int \|P(YY^*) - \hat P \|^2_\F \nu(Y) \rd Y &\leq \| P^0 - \hat P \|^2_\F  + \frac{\log\left(\frac{2}{\epsilon}\right) + D_\KL(\nu, \nu_\theta)}{\lambda} \\
        &+ \left(1 + \frac{\lambda}{m} \right) \int \|P^0 - P(YY^*) \|^2_\F \nu(Y) \rd Y \Bigg \} \geq 1-\frac{\epsilon}{2}\\
        \P \Bigg \{\int \|P^0 - P(YY^*) \|^2_\F \nu(Y) \rd Y &\leq \frac{ \int \|P(YY^*) - \hat P \|^2_\F \nu(Y) \rd Y - \| P^0 - \hat P \|^2_\F +\frac{\log\left(\frac{2}{\epsilon}\right) + D_\KL(\nu, \nu_\theta)}{\lambda}}{1-\frac{\lambda}{m}}\Bigg \} \geq 1-\frac{\epsilon}{2}.    
    \end{align*}}
    The next result is obtained by using a union argument (i.e., replacing the first term in the right-hand side of the second inequality by its upper bound given by the first inequality): for any density $\nu  \in \Pi(\C^{d \times r})$, with probability at least $1-\epsilon$ over the data,
\begin{equation*}
    \int \|P^0 - P(YY^*) \|^2_\F  \nu(Y) \rd Y \leq  \frac{\left(1 + \frac{\lambda}{m} \right) \int \|P^0 - P(YY^*) \|^2_\F \nu(Y) \rd Y   +2 \frac{\log\left(\frac{2}{\epsilon}\right) + D_\KL(\nu, \nu_\theta)}{\lambda}}{1-\frac{\lambda}{m}}.
\end{equation*}
Rewriting this expression for the density that corresponds to the infimum of the right-hand side gives: 
{\small\begin{equation*}
    \int \|P^0 - P(YY^*) \|^2_\F \hat \nu_{\lambda_{\mathrm{new},\theta}}^{\mathrm{BM}}(Y) \rd Y \leq \inf_{\nu  \in \Pi(\C^{d \times r})} \frac{\left(1 + \frac{\lambda}{m} \right) \int \|P^0 - P(YY^*) \|^2_\F \nu(Y) \rd Y   +2 \frac{\log\left(\frac{2}{\epsilon}\right) + D_\KL(\nu, \nu_\theta)}{\lambda}}{1-\frac{\lambda}{m}},
\end{equation*}}
where the posterior density $\hat \nu_{\lambda_{\mathrm{new},\theta}}^{\mathrm{BM}}$ in the left-hand side, for $\lambda_{\mathrm{new}} := \frac{\lambda}{2}(1+\frac{\lambda}{m})$, appears due to \cite[Eq. 10]{Guedj2019}.
\end{proof}

We next need to upper bound the right-hand side of \eqref{eq:final_step_proba}. For this, borrowing from \cite{Dalalyan2020}, we restrict the infimum to a family of priors obtained as translations of our prior $\nu_\theta$, that we write $\bar \nu_{\theta}(Y) = \nu_\theta(Y-\bar Y)$ for some $\bar Y \in \C^{d \times r}$. We derive the following result, which is a direct extension of \cite[Lemma 2]{Dalalyan2020} to the complex setting.
\begin{lemma}
     For any matrix $\bar Y \in \C^{d \times r}$ of rank $p$, let $\bar \nu_\theta$ be the probability density function obtained from the prior $\nu_\theta$ by translation, i.e.,  $\bar \nu_\theta(Y) = \nu_\theta(Y-\bar Y)$. There holds 
    \[ D_{\KL}(\bar \nu_\theta | \nu_\theta) \leq 2p (2d+r+2) \log \left( 1+ \frac{\| \bar Y \|_2}{\theta} \right).\]
\end{lemma}
\begin{proof}This result is a direct extension of \cite[Lemma 2]{Dalalyan2020}. We recall here the proof, extracted from \cite{Dalalyan2020}, for completeness.
    By definition of the KL-divergence, there holds
    \begin{align*}
        D_{\KL}(\bar \nu_\theta | \nu_\theta) &= \int_{\C^{d \times r}} \log \left( \frac{\nu_\theta(Y)}{\bar \nu_\theta(Y)} \right) \nu_\theta(Y) \rd Y \\
        &= \int_{\C^{d \times r}} \log \left( \frac{\nu_\theta(Y)}{\nu_\theta(Y-\bar Y)} \right) \nu_\theta(Y) \rd Y.
    \end{align*}
    
    Note first that, if $p = 0$, the result is trivial. We therefore assume from now $p>0$. We define $A := (\theta^2 I_d + YY^*)^{-1/2}$, and $B := \theta^2 I_d + (Y-\bar Y)(Y-\bar Y)^*$, which are both Hermitian and positive definite\footnote{For a Hermitian positive definite matrix $M$, we write $\log(M)$ and $M^{1/2}$ the principal matrix logarithm and principal matrix square root, respectively. Let $M = U \Lambda U^*$ be an eigenvalue decomposition, with $U \in \C^{d \times d}$ unitary and $\Lambda \in \R^{d \times d}$ diagonal with strictly positive entries. Then, $\log(M) = U \log(\Lambda) U^*$ and $M^{1/2} = U \Lambda^{1/2} U^*$, where the logarithm and square root are applied to the diagonal entries of $\Lambda$. }. It follows that 
    \begin{align}
    2 \log \left( \frac{\nu_\theta(Y)}{\nu_\theta(Y-\bar Y)} \right) &= 2 \log \left( \frac{\det(\theta^2 I_d + YY^*)^{-(2d+r+2)/2}}{\det(\theta^2 I_d + (Y-\bar Y)(Y-\bar Y)^*)^{-(2d+r+2)/2}} \right) \\
    &= (2d+r+2) \log \left( \frac{\det(B)} {\det(A^{-2})} \right) \\
    &= (2d+r+2) \log \left(\det(ABA) \right) \\
    &= (2d+r+2) \sum_{i = 1}^d \log \lambda_i, \label{eq:eigenvalue_log}
    \end{align}
where we denote by $\lambda_i$ the $i$th largest eigenvalue of the matrix $ABA$, with associated eigenvector $u_i \in \C^d$.
We next write 
\begin{align*}
    ABA &= I_d + A\bar Y \bar Y^*A - A \bar Y Y^* A - A Y \bar Y^* A,\\
    &= I_d + A\bar Y (\bar Y -Y)^* A - A Y \bar Y^* A,
\end{align*}
and show that the rank of $ABA-I_d$ is at most the minimum between $2p$ and $d$, which we write $\min(2p,d)$ in short. Indeed, note first that $\range(A \bar Y (\bar Y - Y)^* A) \subseteq A \; \range(\bar Y)$\footnote{$A \calS$ refers to the image of $\calS$ under the application $A$.}, which implies that $\rank(A \bar Y (\bar Y - Y)^* A)$ has dimension at most $p$. Then, note that $\rank(A Y \bar Y^* A) = \rank(A \bar Y Y^* A) \leq p$ using the same argument. It follows that the rank of $ABA-I_d$ is at most $\min(2p,d)$, the matrix $ABA$ has at most $\min(2p,d)$ eigenvalues different from one, and consequently the summation in \eqref{eq:eigenvalue_log} involves at most $\min(2p,d)$ nonzero terms. Let us assume without loss of generality that the nonzero eigenvalues of $ABA-I_d$ are $\lambda_1, \dots, \lambda_K$ with $K \leq \min(2p,d)$ according to the previous discussion. For all $i = 1, \dots, K$, it holds
\begin{align*}
    \lambda_i &= u_i^*(ABA) u_i \\
    &= u_i^* \left(  I_d + A\bar Y (\bar Y -Y)^* A - A Y \bar Y^* A \right) u_i \\
    &= 1 + u_i^* \left( A\bar Y (\bar Y -Y)^* A - A Y \bar Y^* A \right) u_i.
\end{align*}
As $\| (\bar Y - Y)^* A u_i \|^2 = u_i^* A (\bar Y - Y) (\bar Y - Y)^* A u_i = u_i^* A \left( \bar Y \bar Y^* - \bar Y Y^* - Y \bar Y^* + YY^*\right) A u_i = u_i^* \left( A\bar Y (\bar Y -Y)^* A - A Y \bar Y^* A \right) u_i + \| Y^*Au_i \|^2$, it follows that:
\begin{align*}
    \lambda_i &= 1 + \| (\bar Y - Y)^* A u_i \|^2 - \| Y^*Au_i \|^2 \\
    &\leq 1 + (\| \bar Y^*Au_i \| + \| Y^*Au_i \|)^2 - \| Y^*Au_i \|^2 \\
    &= 1 + \| \bar Y^*Au_i \|^2 + \| Y^*Au_i \|^2 + 2 \| \bar Y^*Au_i \| \| Y^*Au_i \| - \| Y^*Au_i \|^2 \\
    &= 1 + \| \bar Y^*Au_i \|^2  + 2 \| \bar Y^*Au_i \| \| Y^*Au_i \| \\ 
    &\leq (1 + \| \bar Y^*Au_i \|)^2,
\end{align*}
where we used in the last inequality the fact that $\| Y^*Au_i \| \leq 1$. Indeed, note that
\[ 1 = u_i^*u_i = u_i^* (\theta^2 I_d + YY^*)^{-1/2}(\theta^2 I_d + YY^*)(\theta^2 I_d + YY^*)^{-1/2} u_i = \theta^2 \|A u_i\|^2 + \| Y^* A u_i\|^2 \geq \| Y^* A u_i\|^2.\]
Using the concavity of the function $\log(1+x^{1/2})$ over $(0,\infty)$, \eqref{eq:eigenvalue_log} gives
\begin{align*}
    2 \log \left(\frac{\nu_\theta(Y)}{\nu_\theta(Y-\bar Y)} \right) &= (2d+r+2) \sum_{i = 1}^{K} \log \lambda_i \\
    &\leq 2 (2d+r+2) \sum_{i = 1}^{K} \log (1+\| \bar Y^* A u_i\|) \\
    &= 2 (2d+r+2) K \sum_{i = 1}^{K}\frac{1}{K} \log (1+(\| \bar Y^* A u_i\|^2)^{1/2}) \\
    &\leq 2K (2d+r+2)  \log \left(1+ \left(\frac{1}{K} \sum_{i = 1}^{K} \| \bar Y^* A u_i\|^2\right)^{1/2} \right).
\end{align*}
Finally, using the fact that the eigenvectors $u_1, \dots u_{K}$ are orthonormal and that $A \preceq \theta^{-1} I_d$\footnote{Note indeed that the eigenvalues of $A$ are given by $((\theta^2+\sigma_i(Y)^2)^{-1/2})_{i = 1, \dots, r}$ with $\sigma_i(Y)$ the $i$th singular value of $Y$. Since $(\theta^2+\sigma_i(Y)^2)^{-1/2} \leq \theta$ for all  $\sigma_i(Y) > 0$, it follows that $A \preceq \theta^{-1} I_d$.}, and writing $U = [u_1, \dots, u_{K}]$, there holds: 
\[ \left( \sum_{i = 1}^{K} \| \bar Y^* A u_i\|^2\right)^{1/2}  = \| \bar Y^* A U\|_\F \leq \sqrt{K} \|\bar Y^* A U\|_2 \leq \sqrt{K} \|\bar Y\|_2 \|AU\|_2 \leq \frac{\sqrt{K}}{\theta} \|\bar Y\|_2. \]
The first inequality comes from \cite[Result 8.5.2(4e)]{Lutkepohl1996}, the second results from the submultiplicativity of the spectral norm, and the last one comes from the unitarily invariance of the spectral norm and from the matrix inequality $A \preceq \theta^{-1} I_d$. As a result, and since $K = \min(2p,d)$, it follows that
\[ \log \left(\frac{\nu_\theta(Y)}{\nu_\theta(Y-\bar Y)} \right) \leq 2p (2d+r+2) \log \left( 1+ \frac{\| \bar Y \|_2}{\theta} \right).\]
\end{proof}

\begin{lemma} \label{lem:upper_bound_proba_gap}
    For any $\bar Y \in \C^{d \times r}$ of rank $p$, let $\bar \nu_{\theta}$ be the shifted prior $\bar \nu_{\theta}(Y) = \nu_\theta(Y-\bar Y)$. Then, there holds:  
    \[ \int \| P(YY^*)-P^0\|_\F^2 \bar \nu_{\theta}(Y) \rd Y \leq d 6^{n/2} (a + \sqrt{2dr} \theta  b + 2dr \theta^2),\]
with $a := \|Y^0-\bar Y \|_\F \|Y^0+\bar Y \|_\F$ and $b:= \|Y^0-\bar Y \|_\F+ \|Y^0+\bar Y \|_\F$.
\end{lemma}

\begin{proof}

Recall that, for all $Y \in \C^{d \times r}$, the entries of the matrices $P(YY^*)$ and $P^0$ are probabilities, and therefore lie in the interval $[0,1]$. As a result, the matrix $M := P(YY^*)-P^0$ satisfies $M_{i,j} \in [-1,1]$, hence $\| M\|_\F^2 = \sum_{i,j = 1}^d M_{i,j}^2 \leq d^2$, which implies that $\| M \|_\F \leq d$. It follows that for all $Y \in \C^{d \times r}$, 
    \begin{equation} \label{eq:norm_fro}
        \| P(YY^*)-P^0\|_\F^2 \leq d \| P(YY^*)-P^0\|_\F.
    \end{equation} 
According to \cite[Proof of Lemma 5]{Alquier2013}, there holds for all $Y \in \C^{d \times r}$
\begin{equation}\label{eq:norm_fro_in_rho_space}
    \| P(YY^*-Y^0Y^{0*})\|_\F \leq  6^{n/2} \|YY^*-Y^0Y^{0*}\|_\F.
\end{equation} 
Then, write 
\[ YY^*-Y^0Y^{0*} = \frac{1}{2} \left( (Y-Y^0) (Y+Y^0)^* + (Y+Y^0) (Y-Y^0)^*\right), \]
and note that, by the triangle inequality and submultiplicativity of the Frobenius norm \cite[p.342]{HornJohnson2013}, there holds: 
\begin{equation} \label{eq:fro_submultiplicative}
     \|YY^*-Y^0Y^{0*}\|_\F \leq \| Y-Y^0\|_\F \| Y+Y^0\|_\F.
\end{equation} 
Using again the triangle inequality, 
\begin{equation} \label{eq:triangle}
 \begin{aligned} 
\|Y^0-Y \|_\F &\leq \|Y^0-\bar Y \|_\F + \| \bar Y - Y \|_\F\\
\|Y^0+Y \|_\F &\leq \|Y^0+\bar Y \|_\F + \| Y - \bar Y \|_\F.
\end{aligned} 
\end{equation}
Combining \eqref{eq:norm_fro}, \eqref{eq:norm_fro_in_rho_space}, \eqref{eq:fro_submultiplicative} and \eqref{eq:triangle} gives:
\begin{align*}\| P(YY^*)-P^0\|_\F^2 &\leq  6^{n/2}d \left(\|Y^0-\bar Y \|_\F + \| \bar Y - Y \|_\F \right)\left(\|Y^0+\bar Y \|_\F + \| Y -\bar Y \|_\F\right) \\
& \leq 6^{n/2}d \left(\|Y^0-\bar Y \|_\F \|Y^0+\bar Y \|_\F + \|Y-\bar Y \|_\F (\|Y^0-\bar Y \|_\F + \|Y^0+\bar Y \|_\F) + \|Y-\bar Y \|_\F^2\right),
\end{align*}
so that
\begin{align*}
    \int \| P(YY^*)-P^0\|_\F^2 \nu_{\theta}(Y-\bar Y) \rd Y \leq &6^{n/2}d (\|Y^0-\bar Y \|_\F \|Y^0+\bar Y \|_\F \\ 
    &+ (\|Y^0-\bar Y \|_\F + \|Y^0+\bar Y \|_\F) \left(\int \|Y-\bar Y \|_\F \nu_{\theta} (Y-\bar Y) \rd Y \right) \\
    &+ \int \|Y-\bar Y \|_\F^2 \nu_{\theta}(Y-\bar Y) \rd Y 
    ).
\end{align*}
We compute
\[ \int \|Y\|_\F \nu_\theta(Y) \d Y = \left( \left( \int \|Y\|_\F \nu_\theta(Y) \d Y \right)^{2} \right)^{1/2} \leq \left(\int \|Y\|^2_\F \nu_\theta(Y) \d Y \right)^{1/2} = \sqrt{2rd} \theta,\] 
where the inequality is Jensen's inequality (stating that for any convex function $\phi$, there holds $\phi(\mathbb{E}(X)) \leq \mathbb{E}(\phi(X))$, for the convex function $x \mapsto x^2$, and where the last equality comes from \Cref{lem:multivariatestudent}.
Similarly, 
\begin{align*}
    \int \|Y\|_\F^2 \nu_\theta(Y) \d Y = 2dr \theta^2.
\end{align*}
The claim follows.
\end{proof}

This allows us to prove \Cref{thm:pac}, that we recall next. 

\paragraph{Theorem 1:}
 Let $\lambda = 3m/8$, and let $\bar Y \in \C^{d \times r}$ satisfy $ \|Y^0-\bar Y \|_\F \|Y^0+\bar Y \|_\F \leq 3^{-3n/2} 2^{-n/2}/m$. Let $p$ be the rank of $\bar Y$. For any $\epsilon \in ]0,1[$, there holds with probability $1-\epsilon$
\begin{align*}
    \|\hat \rho_{\lambda, \theta}^{\mathrm{BM}}  - \rho^0 \|^2_\F  &\leq \frac{3}{N_{\tot}} \left(3^{3n/4} 2^{(n+6)/4} (r + \sqrt{r} \|\bar Y\|_\F) + \frac{2r}{m}+1\right) \\
    &+\frac{3^n 8}{ 2^n  N_{\tot}} \left( \log\left(\frac{2}{\epsilon}\right) + 2p (2^{n+1}+r+2) \log \left( 1+ \frac{\| \bar Y \|_2}{\theta} \right) \right).
\end{align*} 

\begin{proof}
Note first that, according to Jensen's inequality, there holds: 
\[ \| \hat \rho_{\lambda_{\mathrm{new}},\theta}^{\mathrm{BM}} - \rho^0 \|^2_\F \leq \int \|YY^*-Y^0Y^{0*} \|^2_\F \hat \nu_{ \lambda_{\mathrm{new}},\theta}^{\mathrm{BM}}(Y) \d Y \leq \frac{1}{2^n} \int \| P(YY^*) - P^0\|_\F^2  \hat \nu_{\lambda_{\mathrm{new},\theta}}^{\mathrm{BM}}(Y) \d Y,  \]
with $\lambda_{\mathrm{new}} = \frac{\lambda}{2}(1+\frac{\lambda}{m})$ (see the proof of \Cref{lem:Catoni_application}), and where the second inequality comes from \cite[Proof of Lemma 5]{Mai2017}.
 \Cref{lem:Catoni_application} and \Cref{lem:upper_bound_proba_gap} ensure that, with probability $1-\epsilon$ for $\epsilon \in ]0,1[$,
\begin{align*}
\| \hat \rho_{\lambda_{\mathrm{new}},\theta}^{\mathrm{BM}}  - \rho^0 \|^2_\F &\leq \inf_{\bar Y \in \C^{d \times r}}  \frac{6^{n/2} d \left(1 + \frac{\lambda}{m} \right) (a + \sqrt{2dr} \theta  b + 2dr \theta^2) +2 \frac{\log\left(\frac{2}{\epsilon}\right) + 2p (2d+r+2) \log \left( 1+ \frac{\| \bar Y \|_2}{ \theta} \right)}{ \lambda}}{(1-\frac{\lambda}{m}) 2^n} \\
&\leq \inf_{\bar Y \in \C^{d \times r}} N_{\tot} \frac{6^{n/2} d \left(1 + \frac{\lambda}{m} \right) (a + \sqrt{2dr} \theta  b + 2dr \theta^2) +2 \frac{\log\left(\frac{2}{\epsilon}\right) + 2p (2d+r+2) \log \left( 1+ \frac{\| \bar Y \|_2}{ \theta} \right)}{ \lambda}}{(1-\frac{\lambda}{m}) N_{\tot} 2^n}.
\end{align*}
Let us recall that $d = 2^n$ and $N_{\tot} = 3^n m$. Choosing $\theta \propto 6^{-3n/4}/m$ gives with probability $1-\epsilon$
\begin{align*}
    \|\hat \rho_{\lambda_{\mathrm{new}},\theta}^{\mathrm{BM}}  - \rho^0 \|^2_\F  &\leq \inf_{\bar Y  \in \C^{d \times r}} \frac{m\left(1 + \frac{\lambda}{m} \right)}{N_{\tot} \left(1 - \frac{\lambda}{m} \right)} \left(a 3^{3n/2}2^{n/2} +\frac{\sqrt{2r}}{m} b 3^{3n/4} 2^{n/4} + \frac{2r}{m^2}\right) \\
    &+\frac{2m 3^n}{ 2^n\lambda (1-\frac{\lambda}{m}) N_{\tot}} \left( \log\left(\frac{2}{\epsilon}\right) + 2p (2^{n+1}+r+2) \log \left( 1+ \frac{\| \bar Y \|_2}{\theta} \right) \right).
\end{align*} 
Recall that $a = \|Y^0-\bar Y \|_\F \|Y^0+\bar Y \|_\F$ and note that
\[  b = \|Y^0-\bar Y \|_\F + \|Y^0+\bar Y \|_\F \leq 2 (\|Y^0\|_\F + \|\bar Y \|_\F) \leq (2\sqrt{r} + 2 \|\bar Y\|_\F).\]
Choose $\bar Y$ such that $a = 3^{-3n/2} 2^{-n/2}/m$ (observe indeed that $a$ can be made arbitrarily small by choosing $\bar Y$ close enough to $Y^0$), it follows that, with probability $1-\epsilon$, 
\begin{align*}
    \|\hat \rho_{\lambda_{\mathrm{new}},\theta}^{\mathrm{BM}}  - \rho^0 \|^2_\F  &\leq \frac{m\left(1 + \frac{\lambda}{m} \right)}{N_{\tot} \left(1 - \frac{\lambda}{m} \right)} \left(\frac{1}{m} +\frac{\sqrt{2r}}{m}(2\sqrt{r} + 2 \|\bar Y\|_\F) 3^{3n/4} 2^{n/4} + \frac{2r}{m^2}\right) \\
    &+\frac{2m 3^n}{ 2^n\lambda (1-\frac{\lambda}{m}) N_{\tot}} \left( \log\left(\frac{2}{\epsilon}\right) + 2p (2^{n+1}+r+2) \log \left( 1+ \frac{\| \bar Y\|_2}{\theta} \right) \right).
\end{align*} 

For large values of $n$, the right-hand side is dominated by the second term (note indeed that the first term of the right-hand side is negligible for large values of $n$ due to the exponential dependency in $n$ of $N_{\mathrm{tot}}$), which is minimized for $\lambda = m/2$ (corresponding to $\tilde \lambda_{\mathrm{new}} = 3m/8)$, leading to
\begin{align*}
    \|\hat \rho_{\lambda_{\mathrm{new}},\theta}^{\mathrm{BM}}  - \rho^0 \|^2_\F  &\leq \frac{3}{N_{\tot}} \left(3^{3n/4} 2^{(n+6)/4} (r + \sqrt{r} \|\bar Y\|_\F) + \frac{2r}{m}+1\right) \\
    &+\frac{3^n 8}{ 2^n  N_{\tot}} \left( \log\left(\frac{2}{\epsilon}\right) + 2p (2^{n+1}+r+2) \log \left( 1+ \frac{\| \bar Y \|_2}{\theta} \right) \right)
\end{align*} 
with probability $1-\epsilon$. The result follows.
\end{proof}

\subsection{Proof of Lemma 2}

\paragraph{Lemma 2:}
For all $M,N \in \C^{d \times d}$ Hermitian positive semidefinite, there holds $\tr(MN) =  \tr(\psi(M) \psi(N))$ and $\det(M) = \sqrt{2^d \det(\psi(M))}$.

\begin{proof}
Note first that, since $M$ is Hermitian positive semidefinite, $M^* = M^{R \top} -i M^{I \top} = M$ which implies that $M^R = M^{R \top}$ and $M^I = - M^{I \top}$. In other words, $M^R$ (and $N^R$) are symmetric, while $M^I$ (and $N^I$) are skew-symmetric. We compute 
    \[ \tr(\psi(M) \psi(N)) = \frac{1}{2}(\tr(M^RN^R-M^IN^I) + \tr(-M^I N^I + M^R N^R)) = \tr(M^R N^R)-\tr(M^I N^I). \]
    On the other hand, 
    \begin{align*}
        \tr(MN) &= \tr((M^R+i M^I)(N^R+i N^I) = \tr(M^R N^R) - \tr(M^I N^I) + i \tr(M^R N^I + M^I N^R) \\
        &= \tr(M^R N^R) - \tr(M^I N^I),
    \end{align*}
    where the second equality comes from the fact that the inner product between a symmetric and a skew-symmetric matrix is zero; this proves the first claim.

 For the second claim, note that Hermitian matrices are diagonalisable by unitary transformation, and that their determinant is the product of their eigenvalues. We prove that, for any eigenvalue $\lambda$ of $M$ with multiplicity $m_{\lambda}$, there exists an eigenvalue $\lambda/\sqrt{2}$ of $\psi(M)$ with multiplicity $2m_{\lambda}$. Since $M$ is Hermitian, it admits an orthogonal basis of eigenvectors. Let $v \in \C^d$ be an eigenvector of $M$, with eigenvalue $\lambda$, i.e., $Mv = (M^R v^R - M^I v^I) + i (M^R v^I + M^I v^R) = \lambda (v^R + i v^I)$. Then, 
\[ \frac{1}{\sqrt{2}} \begin{pmatrix}  M^R & -M^I \\ M^I & M^R \\ \end{pmatrix}
\begin{pmatrix} v^R \\ v^I \end{pmatrix} = \frac{\lambda}{\sqrt{2}} \begin{pmatrix} v^R \\ v^I \end{pmatrix}, \]
and
\[ \frac{1}{\sqrt{2}} \begin{pmatrix}  M^R & -M^I \\ M^I & M^R \\ \end{pmatrix}
\begin{pmatrix} v^I \\ -v^R \end{pmatrix} = \frac{\lambda}{\sqrt{2}}  \begin{pmatrix} v^I \\ -v^R \end{pmatrix}.
\]
Thus, the two vectors $(v^R, v^I)$ and $(v^I, -v^R)$ are eigenvectors of $\psi(M)$ with eigenvalue $\lambda/\sqrt{2}$. Since these vectors are orthogonal, the result then simply follows from the definition of the determinant as the product of the eigenvalues (taking into account multiplicities). 

\end{proof}

\end{document}